\documentclass[11pt]{article}
\usepackage[]{amsmath,amssymb,amsfonts,latexsym,amsthm,enumerate,fullpage}

\newtheorem{prelem}{{\bf Theorem}}

\newtheorem*{thm}{Informal Statement}

\newtheorem{theorem}{Theorem}[section]
\newtheorem{corollary}[theorem]{Corollary}
\newtheorem{definition}[theorem]{Definition}
\newtheorem{conjecture}[theorem]{Conjecture}
\newtheorem{claim}[theorem]{Claim}

\newtheorem{lemma}[theorem]{Lemma}
\newtheorem{proposition}[theorem]{Proposition}
\newtheorem{observation}[theorem]{Observation}

\newtheorem{remarka}[theorem]{Remark}

\newtheorem{examplea}[theorem]{Example}
\newenvironment{example}{\begin{examplea}\rm}{\hfill\rule{2mm}{2mm}\end{examplea}}

\def\Ex {{\mathbb E}}
\def\Pr {{\rm Pr}}

\def\bias{\textrm{bias}}

\def\rank{{\rm rank}}
\def\bias{{\rm bias}}

\def\exp{{\mathrm{e}_p}}

\def\x{\mathbf{x}}
\def\X{\mathbf{X}}
\def\Y{\mathbf{Y}}
\def\uu{\mathbf{u}}

\def\a{\mathbf{a}}
\def\P{\mathcal{P}}
\def\Q{\mathcal{Q}}

\def\F{\mathbb{F}}
\def\C{\mathbb{C}}
\def\eps{\varepsilon}
\def\N{\mathbb{N}}

\def\D{\mathbb{D}}
\def\E{\mathbb{E}}

\def\ff{\mathrm{f}}
\renewcommand{\gg}{\mathrm{g}}

\newcommand{\ip}[1]{\langle#1\rangle}
\newcommand{\ignore}[1]{}

\newcommand{\restate}[2]{\medskip
\noindent{\bf #1 (restated).}{\sl #2}}

\title{Higher-order Fourier analysis of $\mathbb{F}_p^n$ and the complexity of systems of linear forms}
\author{Hamed Hatami\\
School of Computer Science, McGill University, Montr\'eal,  Canada\\
hatami@cs.mcgill.ca
\and
Shachar Lovett\thanks{Supported by NSF grant DMS-0835373.}\\
School of Mathematics, Institute of Advanced Study, Princeton, USA\\
slovett@math.ias.edu}

\date{}
\begin{document}
\maketitle

\begin{abstract}
In this article we are interested in the density of small linear structures (e.g. arithmetic progressions)  in subsets $A$ of the group $\mathbb{F}_p^n$. It is possible to express these densities as certain analytic averages involving $1_A$, the indicator function of $A$. In the higher-order Fourier analytic approach, the function $1_A$ is decomposed as a sum $f_1+f_2$ where $f_1$ is structured in the sense that it has a simple higher-order Fourier expansion, and $f_2$ is pseudorandom in the sense that the $k$th Gowers uniformity norm of $f_2$, denoted by $\|f_2\|_{U^k}$, is small for a proper value of $k$.

For a given linear structure,  we find the smallest degree of uniformity $k$ such that assuming that $\|f_2\|_{U^k}$ is sufficiently small, it is possible to discard $f_2$ and replace $1_A$ with $f_1$,  affecting the corresponding analytic average only negligibly. Previously, Gowers and Wolf solved this problem for the case where $f_1$ is a constant function. Furthermore, our main result solves Problem 7.6 in
[W.~T. Gowers and J.~Wolf.
\newblock Linear forms and higher-degree uniformity for functions on
  $\mathbb{F}_p^n$.
\newblock {\em Geom. Funct. Anal.}, 21(1):36--69, 2011]  regarding the analytic averages that involve more than one subset of $\mathbb{F}_p^n$.
\end{abstract}

\newline
{{\sc Keywords:} Gowers uniformity, higher-order Fourier analysis, true complexity;


\newpage
\tableofcontents
\newpage


\section{Introduction \label{sec:intro}}

In additive combinatorics one is often interested in the density of small linear structures (e.g. arithmetic progressions) in subsets of Abelian groups. It is possible to express these densities as certain analytic averages.  For example, consider a subset $A$ of a finite Abelian group $G$. Then the density of the $k$-term arithmetic progressions in $A$ is given by
\begin{equation}
\label{eq:nonDegkAP}
\Ex \left[1_A(X) 1_A(X+Y) \ldots 1_A(X+(k-1)Y)\right],
\end{equation}
where $X$ and $Y$ are independent random variables taking values in $G$ uniformly at random, and $1_A$ is the indicator function of $A$.  More generally, one is often interested in analyzing
\begin{equation}
\label{eq:linearAvgGeneralSets}
\Ex \left[1_{A}\left(\sum_{i=1}^k \lambda_{1,i} X_i
\right)\ldots 1_{A}\left(\sum_{i=1}^k
\lambda_{m,i} X_i\right)\right],
\end{equation}
where $X_1,\ldots, X_k$ are independent random variables taking values in $G$ uniformly at random, $\lambda_{i,j}$ are constants, and $A \subseteq G$. Analyzing averages of this type and understanding the relations between them is the core of many problems and results in additive combinatorics and analytic number theory, and there are theories which are developed for this purpose. The theory of uniformity, initiated by the proof of Szemer\'edi's theorem~\cite{MR0369312}, plays an important role in this area, and it was a major breakthrough when Gowers~\cite{MR1844079} introduced a new notion of uniformity in a Fourier-analytic proof for Szemer\'edi's theorem.

Gowers' work initiated an extension of the classical Fourier analysis, called higher-order Fourier analysis of Abelian groups. In this paper we are only interested in the case where the  group is $\mathbb{F}_p^n$, where $p$ is a fixed prime and $n$ is large. In the classical Fourier-analysis of $\mathbb{F}_p^n$, a function is expressed as a linear combination of the characters of $\F_p^n$, which are exponentials of linear polynomials; that is for $\alpha \in \F_p^n$, the corresponding character is defined
as $\chi_\alpha(x) = \exp(\sum_{i=1}^n \alpha_i x_i)$, where $\exp(m):=e^{\frac{2 \pi i}{p}m}$ for $m \in \F_p$. In higher-order Fourier analysis, the linear polynomials are replaced by higher degree polynomials, and one would like to express a function $f:\F_p^n \to \C$ as a linear combination of the functions $\exp(P)$, where $P$ is a polynomial of a certain degree.

Higher-order Fourier expansions are extremely useful in studying averages that are defined through linear structures. To analyze the average in (\ref{eq:linearAvgGeneral}), one usually decomposes the function $1_A$ as $f_1+f_2$, where $f_1$ has a simple higher-order Fourier expansion, while $f_2$ is ``quasirandom'' meaning that it shares certain properties with a random function, and can be discarded as random noise. More precisely, $\|f_1\|_\infty \le 1$ and there is a small constant $C$ such that $f_1=\sum_{i=1}^C c_i \exp(P_i)$ where $c_i$ are constants and $P_i$ are low degree polynomials, and $f_2$ is quasirandom in the sense that for some proper constant $k$, its $k$-th Gowers uniformity norm $\|f_2\|_{U^k}$ is small.

The $U^k$ norms increase as $k$ increases, and thus the condition that $\|f_2\|_{U^k}$ is small becomes stronger. Therefore a question  arises naturally: Given the average (\ref{eq:linearAvgGeneralSets}), what is the smallest $k$ such that under the assumption that $\|f_2\|_{U^k}$ is sufficiently small in a decomposition $1_A=f_1+f_2$, one can discard $f_2$, affecting the average only negligibly? This question was answered by Gowers and Wolf~\cite{gowers-wolf-2010} in the case that $f_1$ is a constant function, provided that the field size $p$ is not too small.
In this work we extend their result to the case where $f_1$ is an arbitrary bounded function.

More concretely, a linear form $L=(\lambda_1,\ldots,\lambda_k) \in \F_p^k$ maps every $\x=(x_1,\ldots,x_k) \in (\F_p^n)^k$ to $L(\x)=\sum_{i=1}^k \lambda_i x_i \in \F_p^n$. Let $\mathbb{D}$ denote the complex unit disk  $\{z \in \C: |z| \le 1\}$.
Gowers and Wolf~\cite{MR2578471} defined the {\em true complexity} of a system of linear forms $\mathcal{L}=\{L_1,\ldots,L_m\}$ as the minimal $d \ge 1$ such that the following holds: for
every $\eps>0$, there exists $\delta>0$ such that if $f:\F_p^n \to \D$ is a  function
with $\|f - \E[f]\|_{U^{d+1}} \le \delta$, then
$$
\left|\E_{\X \in (\F_p^n)^k}\left[\prod_{i=1}^m f(L_i(\X)) \right] - \E[f]^m \right|\le \eps.
$$
That is, as long as $\|f - \E[f]\|_{U^{d+1}}$ is small enough, we can approximate $f$ by the constant function $\E[f]$, affecting the average  only negligibly.

Gowers and Wolf~\cite{gowers-wolf-2010} fully characterized the true complexity of systems of linear forms\footnote{
Their result in fact requires the field size $p$ not to be too small. Our results share the same requirement.}. Let $L=(\lambda_1,\ldots,\lambda_k) \in \F_p^k$ be a linear form in $k$ variables. The $d$-th tensor power of $L$ is given by
$$
L^d = \left(\prod_{j=1}^d \lambda_{i_j}: i_1,\ldots,i_d \in [k]\right) \in \F_p^{k^d}.
$$
Gower and Wolf~\cite{gowers-wolf-2010} proved the following result, which characterizes the true complexity of a system of linear form as a simple linear algebraic condition.

\begin{thm}[Theorem~\ref{thm:true_compltexity_characterization}]
Provided that $p$ is sufficiently large, the true complexity of a system $\mathcal{L}=\{L_1,\ldots,L_m\}$ of linear forms is the minimal $d \ge 1$ such that $L_1^{d+1},\ldots,L_m^{d+1}$ are linearly independent.
\end{thm}

In this work, we show that the true complexity in fact allows to approximate $f$ by any other bounded function $g$, as long as $\|f-g\|_{U^{d+1}}$ is small. This removes the requirement that $g$ is a constant function, which was present in the work of Gowers and Wolf. We already mentioned that any bounded function $f$ can be decomposed as $f=f_1+f_2$ where
$f_1$ is ``structured'' and $f_2$ is ``quasirandom'' (see Theorem~\ref{thm:decompose}); our result thus shows that we can approximate $f$ by $f_1$ without affecting averages significantly. In the context of sets, Gowers and Wolf's result allows one to handle only {\em uniform} sets $A$, for which $1_A - \E[1_A]$ is pseudorandom, while our result allows one to handle {\em all} sets.

\begin{thm}[Theorem~\ref{thm:avg-approx-funcs}]
Let $\mathcal{L}=\{L_1,\ldots,L_m\}$ be a system of linear forms of true complexity $d$. If $p$ is sufficiently large, then for every $\eps>0$, there exists $\delta>0$ such that the following holds. Let $f,g:\F_p^n \to \D$ be  functions
such that $\|f - g\|_{U^{d+1}} \le \delta$. Then
$$
\left|\Ex_{\X \in (\F_p^n)^k} \left[\prod_{i=1}^m f(L_i(\X))\right] - \Ex_{\X \in (\F_p^n)^k} \left[\prod_{i=1}^m g(L_i(\X))\right]\right| \le \eps.
$$
\end{thm}

More generally, one may consider averages over several functions. The average in (\ref{eq:linearAvgGeneralSets}) is the probability that for every $j \in \{1,\ldots,m\}$, the linear combination $\sum_{i=1}^k \lambda_{j,i} X_i$ belongs to $A$. A more general case is the ``off-diagonal'' case, where instead of one subset $A$ there are $m$ subsets $A_1,\ldots,A_m \subseteq \mathbb{F}_p^n$, and one is interested in estimating the probability that for every $j \in \{1,\ldots,m\}$, we have $\sum_{i=1}^k \lambda_{j,i} X_i \in A_j$. Similar to the diagonal case, this can be expressed as an analytic average, and then one can decompose each function $1_{A_i}$ into structured and pseudorandom parts $1_{A_i}=g_i+h_{i}$, and once again the question arises of what level of uniformity suffices for discarding the pseudorandom parts without affecting the average significantly. Similar to the diagonal case, Gowers and Wolf~\cite{gowers-wolf-2010} resolved this problem when all $g_i$ are constant functions. In this work we extend also the off-diagonal case to the general case where $g_i$ can be arbitrary bounded functions.

\begin{thm}[Theorem~\ref{thm:avg-approx-funcs}]
Let $\mathcal{L}=\{L_1,\ldots,L_m\}$ be a system of linear forms of true complexity $d$. If $p$ is sufficiently large, then for every $\eps>0$, there exists $\delta>0$ such that the following holds. Let $f_i,g_i:\F_p^n \to \D$, $i \in [m]$ be functions
such that $\|f_i - g_i\|_{U^{d+1}} \le \delta$. Then
$$
\left|\Ex_{\X \in (\F_p^n)^k} \left[\prod_{i=1}^m f_i(L_i(\X))\right] - \Ex_{\X \in (\F_p^n)^k} \left[\prod_{i=1}^m g_i(L_i(\X))\right]\right| \le \eps.
$$
\end{thm}

It turns out that a more general phenomena holds, and all of the results above are immediate corollaries of the following theorem, which is the main technical contribution of this work. We show that in order to bound these averages, it suffices to have a single index $i$ such that $L_i^{d+1}$ is linearly independent of $\{L_j^{d+1}: j \ne i\}$ and that $\|f_i\|_{U^{d+1}}$ is small. This in particular resolves a conjecture of Gowers and Wolf~\cite{gowers-wolf-2010}.

\begin{thm}[Theorem~\ref{thm:strong_independence}]
Let $\mathcal{L}=\{L_1,\ldots,L_m\}$ be a system of linear forms. Assume that $L_1^{d+1}$ is not in the linear span of $L_2^{d+1},\ldots,L_m^{d+1}$. If $p$ is sufficiently large, then for every $\eps>0$, there exists $\delta>0$ such that for any functions $f_1,\ldots,f_m:\F_p^n \to \D$ with $\|f_1\|_{U^{d+1}} \le \delta$, we have
$$
\left| \Ex_{\X \in (\F_p^n)^k} \left[\prod_{i=1}^m f_i(L_i(\X)) \right] \right| \le \eps.
$$
\end{thm}

So far, we have only discussed which conditions allow discarding the pseudorandom terms. Note that after removing those terms, one arrives at an average of the  form
\begin{equation}
\label{eq:linearAvgGeneral}
\Ex_{X_1,\ldots,X_k \in \mathbb{F}_p^n}\left[f_1\left(\sum_{i=1}^k \lambda_{1,i} X_i
\right)\ldots f_m\left(\sum_{i=1}^k
\lambda_{m,i} X_i\right)\right],
\end{equation}
where each $f_j$ satisfies $\|f_j\|_\infty \le 1$ and has a simple higher-order Fourier expansion. For these expansions to be useful, one needs some kind of orthogonality or at least an approximation of it. The works of Green and Tao~\cite{GreenTaoFiniteFields} and Kaufman and Lovett~\cite{kaufman-lovett} provide an approximate orthogonality that can be used to analyze averages such as $\Ex_{X \in \mathbb{F}_p^n}[f_1(X)\ldots f_m(X)]$ in a straightforward manner when proper higher-order Fourier expansions of $f_1,\ldots,f_m$ are known. However, it is not a priori clear that these results can be applied to analyze more general averages of the form (\ref{eq:linearAvgGeneral}).
In Lemmas~\ref{lem:LinearFormsBias},~\ref{lem:dependencyShape}~and~\ref{lem:nonhomogen} we prove extensions of the results of Green and Tao~\cite{GreenTaoFiniteFields} that are applicable to such general averages. These extensions allow us to approximate (\ref{eq:avgLinearForms}) with a simple formula in terms of the higher-order Fourier coefficients of $f_1,\ldots,f_m$.

Lemmas~\ref{lem:LinearFormsBias},~\ref{lem:dependencyShape}~and~\ref{lem:nonhomogen} are quite useful, and in fact the proof of our main result,
Theorem~\ref{thm:strong_independence}, heavily relies on them. We also apply these lemmas to prove an invariance result (Proposition~\ref{prop:invariance}), which is one of the key tools in our subsequent paper~\cite{TestingPaper} which studies correlation testing for affine invariant properties on $\F_p^n$.

In the setting of functions on $\mathbb{Z}_N$, recently Green and Tao~\cite{GreenTaoRegularity} established similar results and characterized the true complexity of systems of linear forms.

\paragraph{Paper organization}
We give some basic definitions and notation in Section~\ref{sec:prelim}. We discuss the complexity of systems of linear forms, and formally state our main theorems in Section~\ref{sec:complexity_linear_forms}. We give an overview of higher-order Fourier analysis
in Section~\ref{sec:highFourier}. We prove a strong orthogonality result in Section~\ref{sec:strongOrth}. We then
use these to prove our main result, Theorem~\ref{thm:strong_independence}, in Section~\ref{sec:proof_main}. We conclude with some open problems in Section~\ref{sec:summary}.

\section{Definitions and notations}
\label{sec:prelim}
For a natural number $k$, denote $[k]:=\{1,\ldots,k\}$. The complex unit disk is denoted by $\mathbb{D}=\{z \in \C: |z| \le 1\}$. We will usually use the lower English letters $x,y,z$ to denote elements of $\F_p^n$. For $x \in \F_p^n$, and $i \in [n]$, $x(i)$ denotes the $i$-th coordinate of $x$, i.e. $x=(x(1),\ldots,x(n))$. We frequently need to work with the elements of $(\F_p^n)^k$, which we regard as vectors with $k$ coordinates. These elements are denoted with bold font e.g. $\x=(x_1,\ldots,x_k) \in (\F_p^n)^k$. Capital letters $X$, $Y$, \emph{etc} are used to denote random variables.
For an element $m \in \F_p$, we use the notation $\exp(m):=e^{\frac{2 \pi i}{p}m}$. We denote by $f, g,h$ functions from $\F_p^n$ to $\C$. We denote $n$-variate polynomials over $\F_p^n$ by $P,Q$. 

The \emph{bias} of a function $f:\F_p^n \rightarrow \C$ is defined to be the quantity
\begin{equation}
\label{eq:biasDef}
\bias(f) := \left| \Ex_{X \in \F_p^n}[f(X)] \right|.
\end{equation}
%
The inner product of two functions $f,g:\F_p^n \rightarrow \C$ is defined as
\begin{equation}
\label{eq:corrDef}
\ip{f,g} := \Ex_{X \in \F_p^n}[f(X)\overline{g(X)}].
\end{equation}

A \emph{linear form} in $k$ variables is a vector $L=(\lambda_1,\ldots,\lambda_k) \in \F_p^k$ regarded as a linear function from $V^k$ to $V$, for every vector space $V$ over $\F_p$: If $\x=(x_1,\ldots,x_k) \in V^k$, then $L(\x) := \lambda_1 x_1+\ldots+\lambda_k x_k$. A \emph{system of $m$ linear forms} in $k$ variables is a finite set $\mathcal{L}=\{L_1,\ldots,L_m\}$ of \emph{distinct} linear forms, each in $k$ variables. For a function $f:\F_p^n \rightarrow \C$, and a system of linear forms $\mathcal{L} = \{L_1,\ldots,L_m\}$ in $k$ variables, define
\begin{equation}
\label{eq:avgLinearForms}
t_\mathcal{L}(f) := \Ex\left[  \prod_{i=1}^m f(L_i(\X)) \right],
\end{equation}
where $\X$ is a random variable taking values uniformly in $(\F_p^n)^k$.

\begin{definition}[Homogeneous linear forms]
\label{def:homogen}
A system of linear forms $\mathcal{L} = \{L_1,\ldots,L_m\}$  in $k$ variables is called \emph{homogeneous} if for a uniform random variable $\X \in (\F_p^n)^k$, and every fixed $c \in \F_p^n$,  $(L_1(\X),\ldots,L_m(\X))$ has the same distribution as $(L_1(\X)+c,\ldots,L_m(\X)+c)$.
\end{definition}

We wish to identify two systems of linear forms $\mathcal{L}_0 = \{L_1,\ldots,L_m\}$ in $k_0$ variables, and $\mathcal{L}_1=\{L_1',\ldots,L_m'\}$ in $k_1$ variables if after possibly renumbering the linear forms, $(L_1(\X),\ldots,L_m(\X))$ has the same distribution as $(L'_1(\Y),\ldots,L'_m(\Y))$ where $\X$ and $\Y$ are uniform random variables taking values in $(\F_p^n)^{k_0}$ and $(\F_p^n)^{k_1}$, respectively.  Note that the distribution of $(L_1(\X),\ldots,L_m(\X))$ depends exactly on the linear dependencies between $L_1,\ldots,L_m$, and two systems of linear forms lead to the same distributions  if and only if they have the same linear dependencies.

\begin{definition}[Isomorphic linear forms]
\label{def:linearIsomorphic}
Two systems of linear forms $\mathcal{L}_0 $ and $\mathcal{L}_1$ are \emph{isomorphic} if and only if there exists
a \emph{bijection} from $\mathcal{L}_0$ to $\mathcal{L}_1$  that can be extended to an \emph{invertible} linear transformation $T: {\rm span}(\mathcal{L}_0) \rightarrow {\rm span}(\mathcal{L}_1)$.
\end{definition}

Note that if $\mathcal{L} = \{L_1,\ldots,L_m\}$ is a homogeneous system of linear forms, then $(L_1(\X),\ldots,L_m(\X))$ has the same distribution as $(L_1(\X)+Y,\ldots,L_m(\X)+Y)$, where $Y$ is a uniform random variable taking values in $\F_p^n$ and is independent of $\X$. We conclude with the following trivial observation.

\begin{observation}
\label{obs:canonical}
Every homogeneous system of linear forms is isomorphic to a system of linear forms in which there is a variable that appears with coefficient exactly one in every linear form.
\end{observation}

Next we define the Gowers uniformity norms. They are defined in the more general setting of arbitrary finite Abelian groups.
\begin{definition}[Gowers uniformity norms]
Let $G$ be a finite Abelian group and $f:G \rightarrow \C$.  For an integer $k \ge 1$, the $k$-th Gowers norm of $f$, denoted $\|f\|_{U^k}$ is defined by
\begin{equation}
\label{eq:GowersNorm}
\|f\|_{U^k}^{2^k} := \Ex \left[ \prod_{S \subseteq [k]} \mathcal{C}^{k-|S|} f\left(X + \sum_{i \in S} Y_i\right) \right],
\end{equation}
where $\mathcal{C}$ denotes the complex conjugation operator, and $X,Y_1,\ldots,Y_k$ are independent random variables taking values in $G$ uniformly at random.
\end{definition}
In this article we are only interested in the case where $G =\F_p^n$.
These norms were first defined in~\cite{MR1844079} in the case where $G$ is the group $\mathbb{Z}_N$. Note that $\|f\|_{U^1}=\left|\Ex[f(X)]\right|$, and thus $\|\cdot\|_{U^1}$
is a semi-norm rather than a norm.
The facts that the right-hand side of (\ref{eq:GowersNorm}) is always nonnegative, and that for $k>1$, $\| \cdot \|_{U^k}$ is actually a norm are easy to prove, but certainly not trivial (see~\cite{MR1844079} for a proof).

\section{Complexity of a system of linear forms}
\label{sec:complexity_linear_forms}
Let $\mathcal{L}=\{L_1,\ldots,L_m\}$ be a system of linear forms in $k$ variables.
Note that if $A \subseteq \F_p^n$ and $1_A:\F_p^n \to \{0,1\}$ is the indicator function of $A$, then $t_{\mathcal{L}}(1_A)$ is the probability that $L_1(\X),\ldots,L_m(\X)$ all fall in $A$, where $\X \in (\F_p^n)^k$ is uniformly chosen. Roughly speaking, we say $A \subseteq \F_p^n$ is pseudorandom with regards to $\mathcal{L}$ if
$$
t_{\mathcal{L}}(1_A) \approx \left( \frac{|A|}{p^n} \right)^m;
$$
that is if the probability that all $L_1(\X),\ldots,L_m(\X)$ fall in $A$ is close to what we would expect if $A$ was a random subset of $\F_p^n$ of size $|A|$. Let $\alpha = |A|/p^n$ be the density of $A$, and define $f:=1_A - \alpha$. We have
$$
t_{\mathcal{L}}(1_A)=t_{\mathcal{L}}(\alpha+f)=\alpha^m + \sum_{S \subseteq [m], S \ne \emptyset} \alpha^{m-|S|}  t_{\{L_i: i \in S\}}(f).
$$
So, a sufficient condition for $A$ to be pseudorandom with regards to $\mathcal{L}$ is that $t_{\{L_i: i \in S\}}(f) \approx 0$ for all nonempty subsets $S \subseteq [m]$. Green and Tao~\cite{GreenTaoLinear} showed that a sufficient condition for this to occur is that $\|f\|_{U^{s+1}}$ is small enough, where $s$ is the {\em Cauchy-Schwarz complexity} of the system of linear forms.

\begin{definition}[Cauchy-Schwarz complexity]
\label{dfn:cauchy-schwarz-complexity}
Let $\mathcal{L}=\{L_1,\ldots,L_m\}$ be a system of linear forms. The {\em Cauchy-Schwarz complexity} of $\mathcal{L}$ is the minimal $s$ such that the following holds. For every $1 \le i \le m$, we can partition $\{L_j\}_{j \in [m] \setminus \{i\}}$ into $s+1$ subsets, such that $L_i$ does not belong to the linear span of each such subset.
\end{definition}
The reason for the term {\em Cauchy-Schwarz complexity} is the following lemma due to Green and Tao~\cite{GreenTaoLinear}, whose proof is based on a clever iterative application of the Cauchy-Schwarz inequality.
\begin{lemma}[\cite{GreenTaoLinear}]
\label{lemma:cauchy-schwarz-lemma}
Let $f_1,\ldots,f_m:\F_p \to \D$. Let $\mathcal{L}=\{L_1,\ldots,L_m\}$ be a system of $m$ linear forms in $k$ variables of Cauchy-Schwarz complexity $s$. Then
$$
\left| \Ex_{\X \in (\F_p^n)^k} \left[ \prod_{i=1}^m f_i (L_i(\X)) \right]\right| \le \min_{1 \le i \le m} \|f_i\|_{U^{s+1}}.
$$
\end{lemma}
Note that the Cauchy-Schwarz complexity of any system of $m$ linear forms in which any two linear forms are linearly independent (i.e. one is not a multiple of the other) is at most $m-2$, since we can always partition $\{L_j\}_{j \in [m] \setminus \{i\}}$ into the $m-1$ singleton subsets.

The following is an immediate corollary of Lemma~\ref{lemma:cauchy-schwarz-lemma}.

\begin{corollary}
\label{cor:cauchy-schwarz-complexity-bound}
Let $\mathcal{L}=\{L_1,\ldots,L_m\}$ be a system of linear forms in $k$ variables of Cauchy-Schwarz complexity $s$.
Let $f_i,g_i:\F_p^n  \to \D$ be functions for $1 \le i \le m$. Assume that $\|f_i-g_i\|_{U^{s+1}} \le \frac{\eps}{2 m}$ for all $1 \le i \le m$. Then
$$
\left|\Ex_{\X} \left[\prod_{i=1}^m f_i(L_i(\X))\right] - \Ex_{\X}\left[ \prod_{i=1}^m g_i(L_i(\X)) \right]\right| \le \eps,
$$
where $\X \in (\F_p^n)^k$ is uniform.
\end{corollary}
In particular, if $A \subseteq \F_p^n$ of size $|A|=\alpha p^n$ satisfies $\|1_A - \alpha\|_{U^{s+1}} \approx 0$, then $t_{\mathcal{L}}(1_A) \approx \alpha^m$.


\subsection{The true complexity}
The Cauchy-Schwarz complexity of $\mathcal{L}$ gives an upper bound on $s$, such that if $\|1_A - \alpha\|_{U^{s+1}}$ is small enough, then $A$ is pseudorandom with regards to $\mathcal{L}$. Gowers and Wolf~\cite{MR2578471} defined the {\em true complexity} of a system of linear forms as the minimal $s$ such that the above condition holds for all sets $A$.
\begin{definition}[True complexity~\cite{MR2578471}]
\label{def:trueComplexity}
Let $\mathcal{L}=\{L_1,\ldots,L_m\}$ be a system of linear forms over $\F_p$.
The true complexity of $\mathcal{L}$ is the smallest $d \in \N$ with the following property. For every $\eps>0$,
there exists  $\delta>0$ such that if $f : \F_p^n \rightarrow \D$ satisfies $\|f\|_{U^{d+1}} \le \delta$, then
$$\left| t_{\mathcal{L}}(f) \right|\le  \eps.$$
\end{definition}
An obvious bound on the true complexity is the Cauchy-Schwarz complexity of the system. However, there are cases where this is not tight. Gowers and Wolf~\cite{gowers-wolf-2010} characterized the true complexity of systems of linear forms, assuming the field is not too small. For a linear form $L \in \F_p^k$, let $L^{d} \in \F_p^{k^d}$ be the $d$th tensor power of $L$. That is, if $L=(\lambda_1,\ldots,\lambda_k)$, then
$$
L^{d} = \left(\prod_{j=1}^d \lambda_{i_j}: i_1,\ldots,i_d \in [k]\right) \in \mathbb{F}_p^{k^d}.
$$

\begin{theorem}[Characterization of the true complexity of linear systems, Theorem 6.1 in~\cite{gowers-wolf-2010}]
\label{thm:true_compltexity_characterization}
Let $\mathcal{L}=\{L_1,\ldots,L_m\}$ be a system of linear forms over $\F_p^n$ of Cauchy-Schwarz complexity $s \le p$. The true complexity of $\mathcal{L}$ is the minimal $d$ such that $L_1^{d+1},\ldots,L_m^{d+1}$ are linearly independent over $\F_p$.
\end{theorem}

A natural generalization is to allow for multiple sets. Let $A_1,\ldots,A_m \subseteq \F_p^n$ be sets of densities $\alpha_1,\ldots,\alpha_m$. Let $\mathcal{L}=\{L_1,\ldots,L_m\}$ be a system of linear forms over $\F_p$. We say $A_1,\ldots,A_m$ are pseudorandom with respect to $L_1,\ldots,L_m$ if
$$
\Pr_{\X \in (\F_p^n)^k}[L_1(\X) \in A_1,\ldots, L_m(\X) \in A_m] \approx \alpha_1 \cdot \ldots \cdot \alpha_m.
$$
Analogously to the case of a single set, let $f_i = 1_{A_i} - \alpha_i$. Then a sufficient condition is that for all nonempty subsets $S \subseteq [m]$, we have
$$
\Ex \left[ \prod_{i \in S} f_i(L_i(\X)) \right]\approx 0.
$$

In~\cite{gowers-wolf-2010}, Gowers and Wolf showed that if $\mathcal{L}$ has true complexity $d$ and $\|f_1\|_{U^{d+1}},\ldots,\|f_m\|_{U^{d+1}}$ are small enough, then this stronger condition also holds.

\begin{theorem}[Theorem 7.2 in~\cite{gowers-wolf-2010}]
\label{thm:gowers-wolf-multiple-functions}
Let $\mathcal{L}=\{L_1,\ldots,L_m\}$ be a system of linear forms over $\F_p^n$ of Cauchy-Schwarz complexity  $s \le p$ and true complexity $d$. Then for every $\eps>0$, there exists $\delta>0$ such that the following holds. Let $f_1,\ldots,f_m:\F_p^n \to \mathbb{D}$ be functions such that $\|f_i\|_{U^{d+1}} \le \delta$, for all $1 \le i \le m$. Then for all nonempty subsets $S \subseteq [m]$ we have
$$
\left| \Ex_{\X \in (\F_p^n)^k}\left[ \prod_{i \in S} f_i(L_i(\X)) \right]\right| \le \eps.
$$
\end{theorem}
In particular, Gowers and Wolf used this to derive the following corollary.
\begin{corollary}[Theorem 7.1 in~\cite{gowers-wolf-2010}]
\label{cor:gowers-wolf-approx-bias}
Let $\mathcal{L}=\{L_1,\ldots,L_m\}$ be a system of linear forms of true complexity $d$ and Cauchy-Schwarz
complexity at most $p$. Then for every $\eps>0$, there exists $\delta>0$ such that the following holds.
Let $f_i:\F_p^n \to \D$ for $1 \le i \le m$ be functions such that $\|f_i - \Ex[f_i]\|_{U^{d+1}} \le \delta$. Then
$$
\left|\Ex_{\X}\left[ \prod_{i=1}^m f_i(L_i(\X))\right] - \prod_{i=1}^m \Ex[f_i]\right| \le \eps.
$$
\end{corollary}

Note that Corollary~\ref{cor:gowers-wolf-approx-bias} says  if $\|f_i - \Ex[f_i]\|_{U^{d+1}}$ is small for all $i \in [m]$, then in the average $\Ex_{\X}\left[ \prod_{i=1}^m f_i(L_i(\X))\right]$ one can replace the functions $f_i$ with their expected values $\Ex[f_i]$ causing only a small error of $\eps$. In other words if in the decomposition $f_i= \Ex[f_i] + (f_i-\Ex[f_i])$, the part $(f_i-\Ex[f_i])$ is sufficiently pseudorandom, then it is possible to discard it. As we shall see in Section~\ref{sec:decompos}, for every function $f_i:\mathbb{F}_p^n \to \mathbb{D}$, it is possible to find a ``structured'' function $g_i$, such that $f_i-g_i$ is pseudorandom in the sense that $\|f_i-g_i\|_{U^{d+1}}$ can be made arbitrarily small. However, in the general case the function $g_i$ will not necessarily be  the constant function $\Ex[f_i]$. Hence it is important to obtain a version of Corollary~\ref{cor:gowers-wolf-approx-bias} that can be applied in this general situation. We achieve this in the following theorem,  which qualitatively improves both Corollaries~\ref{cor:cauchy-schwarz-complexity-bound}~and~\ref{cor:gowers-wolf-approx-bias}.

\begin{theorem}
\label{thm:avg-approx-funcs}
Let $\mathcal{L}=\{L_1,\ldots,L_m\}$ be a system of linear forms of true complexity $d$ and Cauchy-Schwarz complexity at most $p$. Then for every $\eps>0$, there exists $\delta>0$ such that the following holds. Let $f_i,g_i:\F_p^n \to \D$
for $1 \le i \le m$ be functions such that $\|f_i - g_i\|_{U^{d+1}} \le \delta$. Then
$$
\left|\Ex_{\X} \left[\prod_{i=1}^m f_i(L_i(\X))\right] - \Ex_{\X} \left[\prod_{i=1}^m g_i(L_i(\X))\right]\right| \le \eps,
$$
where $\X \in (\F_p^n)^k$ is uniform.
\end{theorem}

In fact, we prove a stronger result, from which Theorem~\ref{thm:avg-approx-funcs} follows immediately. As Gowers and Wolf~\cite{gowers-wolf-2010} correctly conjectured, in Theorem~\ref{thm:gowers-wolf-multiple-functions}  the condition that the true complexity of $\mathcal{L}$ is at most $d$ can be replaced by a much weaker condition. It suffices to assume that $L_1^{d+1}$ linearly independent of $L_2^{d+1},\ldots,L_m^{d+1}$. In fact it turns out that more is true\footnote{Gowers and Wolf required that $L_1^{d+1}$ is linearly independent of $L_2^{d+1},\ldots,L_m^{d+1}$, and that all $\|f_1\|_{U^{d+1}},\ldots,\|f_m\|_{U^{d+1}}$ will be bounded by $\delta$.}, and the condition that all of $\|f_1\|_{U^{d+1}},\ldots,\|f_m\|_{U^{d+1}}$ are small also can be replaced by the weaker condition that only $\|f_1\|_{U^{d+1}}$ is small.

\begin{theorem}[Main theorem]
\label{thm:strong_independence}
Let $\mathcal{L}=\{L_1,\ldots,L_m\}$ be a system of linear forms  of Cauchy-Schwarz complexity at most $p$. Let $d \ge 0$, and assume that $L_1^{d+1}$ is not in the linear span of $L_2^{d+1},\ldots,L_m^{d+1}$. Then for every $\eps>0$, there exists $\delta>0$ such that for any functions $f_1,\ldots,f_m:\F_p^n \to \D$ with $\|f_1\|_{U^{d+1}} \le \delta$, we have
$$
\left| \Ex_{\X \in (\F_p^n)^k} \left[\prod_{i=1}^m f_i(L_i(\X)) \right] \right| \le \eps,
$$
where $\X \in (\F_p^n)^k$ is uniform.
\end{theorem}

Theorem~\ref{thm:strong_independence} improves both Lemma~\ref{lemma:cauchy-schwarz-lemma} and Theorem~\ref{thm:gowers-wolf-multiple-functions}:
Let $s \le p$ denote the Cauchy-Schwarz complexity of $\mathcal{L}$ in Theorem~\ref{thm:strong_independence}. Lemma~\ref{lemma:cauchy-schwarz-lemma} requires the stronger condition $\|f_1\|_{U^{s+1}} \le \delta$, and Theorem~\ref{thm:gowers-wolf-multiple-functions} requires two stronger conditions: that $L_1^{d+1},\ldots,L_m^{d+1}$ are linearly independent; and that all $\|f_1\|_{U^{d+1}},\ldots,\|f_m\|_{U^{d+1}}$ are bounded by $\delta$.

\section{Higher-order Fourier analysis}
\label{sec:highFourier}
Although Fourier analysis is a powerful tool in arithmetic combinatorics, there are key questions that cannot be addressed by this method in its classical form. For example in 1953 Roth~\cite{MR0051853} used Fourier analysis to show
that every dense subset of integers contains $3$-term arithmetic progressions. For more than four decades generalizing Roth's Fourier-analytic proof remained an important unsolved problem until finally Gowers in~\cite{MR1844079} introduced an extension of the classical Fourier analysis, which enabled him to obtain such a generalization. The work of Gowers initiated a theory, which has now come to be known as higher-order Fourier analysis. Ever since
several mathematicians  contributed to  major developments in this rapidly growing theory.

This section has two purposes. One is to review the main results that form the foundations of the higher-order Fourier analysis. A second is to establish some new facts that enable us to deal with the averages $t_\mathcal{L}$ conveniently by appealing to higher-order Fourier analysis. The work of Gowers and Wolf~\cite{MR2578471} plays a central role for us, and many ideas in the proofs and the new facts established in this section are hinted by their work.

The characters of $\F_p^n$ are exponentials of linear polynomials; that is for $\alpha \in \F_p^n$, the corresponding character is defined
as $\chi_\alpha(x) = \exp(\sum_{i=1}^n \alpha_i x_i)$. In higher-order Fourier analysis, the linear polynomials $\sum \alpha_i x_i$ are replaced by higher degree polynomials, and one would like to express a function $f:\F_p^n \to \C$ as a linear combination of the functions $\exp(P)$, where $P$ is a polynomial of a certain degree.

Consider a function $f:\F_p^n \rightarrow \C$, and a system of linear forms $\mathcal{L}=\{L_1,\ldots,L_m\}$. The basic properties of characters enable us to express $t_\mathcal{L}(f)$ as a simple formula in terms of the Fourier coefficients of $f$. Indeed, if $f:= \sum_{\alpha \in \F_p^n} \widehat{f}(\alpha) \chi_\alpha$ is the Fourier expansion of $f$, then it is easy to see that
\begin{equation}
\label{eq:avgLinearFourier}
t_\mathcal{L}(f) = \sum \widehat{f}(\alpha_1) \ldots \widehat{f}(\alpha_m),
\end{equation}
where the sum is over all $\alpha_1,\ldots,\alpha_m \in \F_p^n$ satisfying $\sum_{i=1}^m \alpha_i \otimes L_i \equiv 0$. The tools that we develop in this section enables us to obtain simple formulas similar to (\ref{eq:avgLinearFourier}) when Fourier expansion is replaced by a proper higher-order Fourier expansion.

\subsection{Inverse theorems for Gowers uniformity norms \label{sec:inverse}}
We start with some basic definitions.

\paragraph{Polynomials:} Consider a function $\ff: \F_p^n \rightarrow \F_p$. For an element $y \in \F_p^n$, define the \emph{derivative} of $\ff$ in the direction $y$ as $\Delta_y \ff(x)= \ff(x+y)-\ff(x)$. Inductively we define $\Delta_{y_1,\ldots,y_k}\ff=\Delta_{y_k}(\Delta_{y_1,\ldots,y_{k-1}} \ff)$, for directions $y_1,\ldots,y_k \in \F_p^n$. We say that $\ff$ is a \emph{polynomial of degree at most $d$} if for every $y_1,\ldots,y_{d+1} \in \F_p$, we have $\Delta_{y_1,\ldots,y_{d+1}} \ff \equiv 0$. The set of polynomials of degree at most $d$ is a vector space over $\F_p$, which we denote by ${\rm Poly}_d(\F_p^n)$. It is easy to see that the set of \emph{monomials} $x_1^{i_1} \ldots x_n^{i_n}$ where $0 \le i_1,\ldots,i_n < p$ and $\sum_{j=1}^n i_j \le d$ form a basis for ${\rm Poly}_d(\F_p^n)$. So every polynomial $P \in {\rm Poly}_d(\F_p^n)$ is of the from $P(x):=\sum c_{i_1,\ldots,i_n} x_1^{i_1} \ldots x_n^{i_n}$, where the sum is over all $1 \le i_1,\ldots,i_n < p$ with $\sum_{j=1}^n i_j \le d$, and $c_{i_1,\ldots,i_n}$ are elements of $\F_p$. The \emph{degree} of a polynomial $P:\F_p^n \rightarrow \F_p$, denoted by $\deg(P)$, is the smallest $d$ such that $P \in {\rm Poly}_d(\F_p^n)$. A polynomial  $P$ is called \emph{homogeneous} if all monomials with non-zero coefficients in the expansion of $P$ are of degree exactly $\deg(P)$.

\paragraph{Phase Polynomials:} For a function $f:\F_p^n \rightarrow \C$, and a direction $y \in \F_p^n$ define the \emph{multiplicative derivative} of $f$ in the direction of $y$ as $\widetilde\Delta_{y}f(x)=f(x+y)\overline{f(x)}$. Inductively we define $\widetilde{\Delta}_{y_1,\ldots,y_k}f=\widetilde{\Delta}_{y_k}(\widetilde{\Delta}_{y_1,\ldots,y_{k-1}} f)$, for directions $y_1,\ldots,y_k \in \F_p^n$.  A function $f:\F_p^n \rightarrow \C$ is called a \emph{phase polynomial of degree at most $d$}  if for every $y_1,\ldots,y_{d+1} \in \F_p$, we have $\widetilde{\Delta}_{y_1,\ldots,y_{d+1}} f \equiv 1$. We denote the space of all phase polynomials of degree at most $d$ over $\F_p^n$ by $\mathcal{P}_d(\F_p^n)$. Note that for every $\ff: \F_p^n \rightarrow \F_p$, and every $y \in \F_p^n$, we have that
$$\widetilde{\Delta}_y \exp(\ff) = \exp(\Delta_y \ff).$$
This shows that if $\ff \in {\rm Poly}_d(\F_p^n)$, then $\exp(\ff)$ is a phase polynomial of degree at most $d$. The following simple lemma shows that the inverse is essentially true in high characteristics:
\begin{lemma}[Lemma 1.2 in~\cite{tao-2008}]
\label{lem:phasePoly}
Suppose that $0 \le d < p$. Every $f \in \mathcal{P}_d(\F_p^n)$ is of the form $f(x)=\exp(\theta + \ff(x))$, for some $\theta \in \mathbb{R}/\mathbb{Z}$, and $\ff \in {\rm Poly}_d(\F_p^n)$.
\end{lemma}
When $d \ge p$, more complicated phase polynomials arise. Nevertheless obtaining a complete characterization is possible~\cite{tao-blog}.

Now let us describe the relation between the phase polynomials and the Gowers norms. First note that one can express Gowers uniformity norms using multiplicative derivatives:
$$\|f\|_{U^k}^{2^k} = \Ex\left[ \widetilde{\Delta}_{Y_1,\ldots,Y_k}f(X)\right],$$
where $X,Y_1,\ldots,Y_k$ are independent random variables taking values in $\F_p^n$ uniformly. This for example shows that every phase polynomial $g$ of degree at most $d$ satisfies $\|g\|_{U^{d+1}}=1$.

Many basic properties of  Gowers uniformity norms are implied by the Gowers-Cauchy-Schwarz inequality, which is first proved in~\cite{MR1844079} by iterated applications of the classical Cauchy-Schwarz inequality.
\begin{lemma}[Gowers-Cauchy-Schwarz]
\label{lem:GowersCauchyShwarz}
Let $G$ be a finite Abelian group, and consider a family of functions $f_S:G \rightarrow \C$, where $S \subseteq [k]$. Then
\begin{equation}
\left| \Ex \left[\prod_{S \subseteq [k]} \mathcal{C}^{k-|S|} f_S(X + \sum_{i \in S} Y_i) \right]\right|\le \prod_{S \subseteq [k]} \|f_S\|_{U^k},
\end{equation}
where $X,Y_1,\ldots,Y_k$ are independent random variables taking values in $G$ uniformly at random.
\end{lemma}
A simple application of Lemma~\ref{lem:GowersCauchyShwarz} is the following. Consider an arbitrary function $f:G \rightarrow \C$.
Setting $f_\emptyset :=f$ and $f_S:=1$ for every $S \neq \emptyset$ in Lemma~\ref{lem:GowersCauchyShwarz}, we obtain
\begin{equation}
\label{eq:avgVsGowers}
|\Ex[f] | \le \|f\|_{U^k}.
\end{equation}
Equation~(\ref{eq:avgVsGowers}) in particular shows that if $f,g:\F_p^n \to \C$,
then one can bound their inner product by Gowers uniformity norms of $f\overline{g}$:
\begin{equation}
\label{eq:GowersCorrelation}
|\ip{f,g}| \le \|f\overline{g}\|_{U^k}.
\end{equation}
Consider an arbitrary $f:\F_p^n \to \C$ and a phase polynomial $g$ of degree at most $d$. Then for every $y_1,\ldots,y_{d+1} \in \F_p^n$, we have
$$\widetilde{\Delta}_{y_1,\ldots,y_{d+1}} (f g) = (\widetilde{\Delta}_{y_1,\ldots,y_{d+1}} f) (\widetilde{\Delta}_{y_1,\ldots,y_{d+1}}g) =\widetilde{\Delta}_{y_1,\ldots,y_{d+1}} f,$$
which in turn implies that $\|fg\|_{U^{d+1}}=\|f\|_{U^{d+1}}$.  We conclude that
\begin{equation}
\label{eq:GowersCorrelationII}
\sup_{g \in \mathcal{P}_d} \left| \ip{f,g} \right| \le \|f\|_{U^{d+1}}.
\end{equation}
This provides us with a ``direct theorem'' for the $U^{d+1}$ norm:
If  $\sup_{g \in \mathcal{P}_d} \left| \ip{f,g} \right| \ge \eps$, then $\|f\|_{U^{d+1}} \ge \eps$. The following theorem provides the corresponding inverse theorem.
\begin{theorem}[\cite{bergelson-2009,tao-2008,TaoZiegler2010}]
\label{thm:inverse}
Let $d$ be a positive integer. There exists a function $\delta:(0,1] \rightarrow (0,1]$ such that for every $f:\F_p^n \rightarrow \mathbb{D}$, and $\eps>0$,
\begin{itemize}
\item \emph{Direct theorem:} If $\sup_{g \in \mathcal{P}_d} \left| \ip{ f,g } \right| \ge \eps$, then $\|f\|_{U^{d+1}} \ge \eps$.
\item \emph{Inverse theorem:} If $\|f\|_{U^{d+1}} \ge \eps$, then $\sup_{g \in \mathcal{P}_d} \left| \ip{ f,g } \right| \ge \delta(\eps).$
\end{itemize}
\end{theorem}

In the case of $1\le d<p$, Theorem~\ref{thm:inverse} is established by Bergelson, Tao, and Ziegler~\cite{bergelson-2009,tao-2008}. The case $d \ge p$ is established only very recently by Tao and Ziegler~\cite{TaoZiegler2010}.  In the range $1\le d<p$, Lemma~\ref{lem:phasePoly} shows that the phase polynomials of degree at most $d$ can be described using polynomials of degree at most $d$. So Theorem~\ref{thm:inverse} shows that in this case if $\|f\|_{U^{d+1}} \ge \eps$, then there exists a polynomial $\gg:\F_p^n \to \F_p$ of degree at most $d$ such that $|\ip{f, \exp(\gg)}| \ge \delta(\eps)>0$.

\subsection{Decomposition theorems\label{sec:decompos}}

An important application of the inverse theorems is that they imply  ``decomposition theorems''. Roughly speaking these results say that under appropriate conditions, a function $f$ can be decomposed as $f_1+f_2$, where $f_1$ is ``structured'' in some sense that enables one to handle it easily, while $f_2$ is ``quasi-random'' meaning that it shares certain properties with a random function, and can be discarded as random noise. They are discussed in this abstract form in~\cite{GowersSurvey}. In the following we will discuss decomposition theorems that follow from Theorem~\ref{thm:inverse}, but first we need to define polynomial factors on $\F_p^n$.
\begin{definition}[Polynomial factors]
\label{def:polyFactor}
Let $p$ be a fixed prime. Let $P_1,\ldots,P_C \in {\rm Poly}_d(\F_p^n)$. The sigma-algebra on $\F_p^n$ whose atoms are $\{x \in \F_p^n : P_1(x)=a(1),\ldots,P_C(x)=a(C)\}$ for all $a \in \F_p^C$ is called a \emph{polynomial factor of degree at most $d$ and complexity at most $C$}.
\end{definition}
Let $\mathcal{B}$ be a polynomial factor defined by $P_1,\ldots,P_C$. For $f:\F_p^n \to \C$, the conditional expectation of $f$ with respect to $\mathcal{B}$, denoted $\E(f|\mathcal{B}):\F_p^n \to \C$, is
$$
\E(f|\mathcal{B})(x) = \Ex_{\{y \in \F_p^n: P_1(y)=P_1(x),\ldots,P_C(y)=P_C(x)\}}[f(y)].
$$
That is, $\E(f|\mathcal{B})$ is constant on every atom of $\mathcal{B}$, and this constant is the average value that $f$ attains on this atom. A function $g:\F_p^n \to \C$ is $\mathcal{B}$-measurable if it is constant on every atom of $\mathcal{B}$. Equivalently, we can write $g$ as $g(x)=\Gamma(P_1(x),\ldots,P_C(x))$ for some function $\Gamma:\F_p^C \to \C$. The following claim is quite useful, although its proof is immediate and holds for every sigma-algebra.
\begin{observation}
\label{obs:project_factor}
Let $f:\F_p^n \to \C$. Let $\mathcal{B}$ be a polynomial factor defined by polynomials $P_1,\ldots,P_C$.
Let $g:\F_p^n \to \C$ be any $\mathcal{B}$-measurable function. Then
$$
\ip{f,g} = \ip{\E(f|\mathcal{B}),g}.
$$
\end{observation}
The following theorem that follows in a standard manner from Theorem~\ref{thm:inverse} gives a simple decomposition theorem.
\begin{theorem}[Decomposition Theorem~\cite{MR2359469}]
\label{thm:decompose}
Let $p$ be a fixed prime, $0\le d<p$ be an integer,  and $\eps>0$. Given any function $f:\F_p^n \rightarrow \mathbb{D}$, there exists a polynomial factor $\mathcal{B}$ of degree at most $d$ and complexity at most $C_{\max}(p,d,\eps)$ together with a decomposition
$$f=f_1+f_2,$$
where
$$\mbox{$f_1 := \E(f|\mathcal{B})$  and $\|f_2\|_{U^{d+1}} \le \eps$.}$$
\end{theorem}

We sketch the standard proof of Theorem~\ref{thm:decompose} below, as we will need some extensions of it in this paper. For a full proof we refer the reader to~\cite{MR2359469}.

\begin{proof}[Proof sketch]
We create a sequence of polynomial factors $\mathcal{B}_1,\mathcal{B}_2,\ldots$ as follows. Let $\mathcal{B}_1$ be the trivial factor (i.e. $\E(f|\mathcal{B}_1)$ is the constant function $\Ex[f]$). Let $g_i = f - \E(f|\mathcal{B}_i)$. If $\|g_i\|_{U^{d+1}} \le \eps$ we are done. Otherwise by Theorem~\ref{thm:inverse}, since $\|g_i\|_{\infty} \le 2$, there exists a polynomial $P_i \in {\rm Poly}_d(\F_p^n)$ such that $\ip{g_i,\exp(P_i)} \ge \delta(\eps)$. Let $\mathcal{B}_{i+1} = \mathcal{B}_i \cup \{P_i\}$. The key point is that one can show that
$$
\|g_{i+1}\|_2^2 \le \|g_i - \ip{g_i,\exp(P_i)}\|_2^2 \le \|g_i\|_2^2 - \delta(\eps).
$$
Thus, the process must stop after at most $1/\delta(\eps)$ steps.
\end{proof}

%
%
Suppose that the factor $\mathcal{B}$ is defined by $P_1,\ldots,P_C \in {\rm Poly}_d(\F_p^n)$. Assume that $f_1(x)=\Gamma(P_1(x),\ldots,P_C(x))$. Using the Fourier decomposition of $\Gamma$, we can express $f_1$ as
\begin{equation}
\label{eq:HigherFourier}
f_1(x) = \sum_{\gamma \in \F_p^C} \widehat{\Gamma}(\gamma) \exp\left(\sum_{i=1}^C \gamma(i) P_i(x)\right).
\end{equation}
Note that for every $\gamma \in \F_p^C$, $\sum_{i=1}^C \gamma(i) P_i(x) \in {\rm Poly}_d(\F_p^n)$. So~(\ref{eq:HigherFourier}) gives an expansion for $f_1$ which is similar to the Fourier expansion, but instead of characters $\exp(\sum \alpha(i) x_i)$, we have exponential functions $\exp\left(\sum_{i=1}^C \gamma(i) P_i(x)\right)$ which have polynomials of degree $d$ in the powers instead of linear functions. For this decomposition to be useful similar to the Fourier expansion, one needs some kind of orthogonality for the functions appearing in the expansion.

\begin{definition}[Bias]
The bias of a polynomial $P \in {\rm Poly}_d(\F_p^n)$ is defined as
$$
\bias(P) := \bias(\exp(P)) = |\Ex_{X \in \F_p^n}[\exp(P(X))]|.
$$
\end{definition}
We shall refine the set of polynomials $\{P_1,\ldots,P_C\}$ to obtain a new set of polynomials with the desired ``approximate orthogonality'' properties. This will be achieved through the notion of the \emph{rank} of a set of polynomials.

\begin{definition}[Rank]  We say a set of polynomials $\P=\{P_1,\ldots,P_t\}$ is of rank greater than $r$, and denote this by $\rank(\P) > r$ if the following holds. For any non-zero $\alpha=(\alpha_1,\ldots,\alpha_t) \in \F_p^t$, define $P_{\alpha}(x) := \sum_{j=1}^t \alpha_j P_j(x)$. For $d := \max \{\deg(P_j):\alpha_j \ne 0\}$, the polynomial $P_{\alpha}$ cannot be expressed as a function of $r$ polynomials of degree at most $d-1$. More precisely, it is not possible to find $r$ polynomials $Q_1,\ldots,Q_r$ of degree at most $d-1$, and a function $\Gamma:\F_p^r \to \F_p$ such that
$$
P_\alpha(x)=\Gamma(Q_1(x),\ldots,Q_r(x)).
$$
The rank of a single polynomial $P$ is defined to be $\rank(\{P\})$.
\end{definition}
The \emph{rank} of a polynomial factor is the rank of the set of polynomials defining it.
The following lemma follows from the definition of the rank. For a proof see~\cite{GreenTaoFiniteFields}.

\begin{lemma}[Making factors high-rank]
\label{lem:regularity}
Let $r:\N \rightarrow \N$ be an arbitrary growth function. Then there is another function  $\tau:\N \rightarrow \N$ with the following property. Let $\mathcal{B}$ be a polynomial factor with complexity at most $C$. Then there is a refinement $\mathcal{B}'$ of $\mathcal{B}$  with complexity at most $C' \le \tau(C)$ and rank at least $r(C')$.
\end{lemma}
%
The following theorem due to Kaufman and Lovett~\cite{kaufman-lovett}  connects the notion of the rank to the bias of a polynomial. It was proved first by Green and Tao~\cite{GreenTaoFiniteFields} for the case $d<p$, and then extended by Kaufman and Lovett~\cite{kaufman-lovett} for the general case.
\begin{theorem}[Regularity~\cite{kaufman-lovett}]
\label{thm:regularity}
Fix $p$ prime and $d \ge 1$. There exists a function $r_{p,d}:(0,1] \rightarrow \mathbb{N}$ such that the following holds. If $P:\F_p^n \rightarrow \F_p$ is a polynomial of degree at most $d$ with $\bias(P) \ge \eps$, then $\rank(P) \le r_{p,d}(\eps)$.
\end{theorem}
%
%
%
Combining Lemma~\ref{lem:regularity} with Theorem~\ref{thm:decompose}, it is possible to obtain a strong decomposition theorem.
\begin{theorem}[Strong Decomposition Theorem~\cite{MR2359469}]
\label{thm:decompose_strong}
Let $p$ be a fixed prime, $0\le d<p$ be an integer,  $\delta>0$, and let $r:\N \rightarrow \N$ be an arbitrary growth function. Given any function $f:\F_p^n \rightarrow \mathbb{D}$, there exists a decomposition
$$
f=f_1+f_2,
$$
such that
$$\mbox{$f_1:=\E(f|\mathcal{B})$, \qquad $\|f_2\|_{U^{d+1}} \le \delta$,}$$
where  $\mathcal{B}$ is a polynomial factor of degree at most $d$, complexity  $C \le C_{\max}(p,d,\delta,r(\cdot))$, and rank at least $r(C)$.
\end{theorem}

We sketch the proof below. For a full proof we refer the reader to~\cite{MR2359469}.

\begin{proof}[Proof sketch]
The proof follows the same steps as the proof of Theorem~\ref{thm:decompose}, except that at each step, we regularize each polynomial factor $\mathcal{B}_i$ to obtain $\mathcal{B'}_i$, and set $\mathcal{B}_{i+1} = \mathcal{B'}_i \cup \{P_i\}$. The only new insight is that as $\mathcal{B'}_i$ is a refinement of $\mathcal{B}_i$ we have
$$
\|f - \E(f|\mathcal{B'}_i)\|_2 \le \|f - \E(f|\mathcal{B}_i)\|_2.
$$
\end{proof}
Note that Theorem~\ref{thm:decompose_strong} guarantees a strong approximate orthogonality. For every fixed function $\omega:\mathbb{N} \rightarrow \mathbb{N}$, by taking $r(\cdot)$ to be a sufficiently fast growing function, one can guarantee that the polynomials $P_1,\ldots,P_C$ that define the factor $\mathcal{B}$ have the property
\begin{equation}
\label{eq:approxOrth}
\left|\Ex \left[\exp\left(\sum_{i=1}^C \gamma(i) P_i(X)\right)\right]\right| \le 1/\omega(C),
\end{equation}
for all nonzero $\gamma \in \F_p^C$. That is, the polynomials can be made ``nearly orthogonal'' to any required precision.

%
%
The decomposition theorems stated to far referred to a single function. In this paper we require  decomposition theorems which relate to several functions with a single polynomial factor. The proofs can be adapted in a straight-forward manner to prove the next result.
\begin{lemma}[Strong Decomposition Theorem - multiple functions]
\label{lemma:decompose_multiple_func}
Let $p$ be a fixed prime, $0\le d<p$ and $m$ be integers, let $\delta>0$, and let $r:\N \rightarrow \N$ be an arbitrary growth function. Given every set of functions $f_1,\ldots,f_m:\F_p^n \rightarrow \mathbb{D}$, there exists a decomposition of each $f_i$ as
$$
f_i=h_{i}+h'_{i},
$$
such that
$$\mbox{$h_{i}:=\E(f_i|\mathcal{B})$, \qquad $\|h'_{i}\|_{U^{d+1}} \le \delta$,}$$
where  $\mathcal{B}$ is a polynomial factor of degree at most $d$, complexity  $C \le C_{\max}(p,d,\delta,m,r(\cdot))$ and rank at least $r(C)$. Furthermore we can assume that $\mathcal{B}$ is defined by homogeneous polynomials.
\end{lemma}
%
%
%

\section{Strong orthogonality}
\label{sec:strongOrth}
Let $\mathcal{B}$ be a polynomial factor defined by polynomials $P_1,\ldots,P_C$, and let $f:\F_p^n \to \D$ be a $\mathcal{B}$-measurable function. We saw in~\eqref{eq:HigherFourier} that $f$ can be expressed as a linear combination of $\exp(\sum_{i=1}^C \gamma(i) P_i(x))$, for $\gamma \in \F_p^C$. Furthermore if we require the polynomials to be of  high rank, then we obtain an approximate orthogonality as in~\eqref{eq:approxOrth}.
This  approximate orthogonality is sufficient for analyzing averages such as $\Ex[f_1(X)f_2(X)\ldots f_m(X)]$, where $f_1,\ldots,f_m$ (not necessarily distinct) are all measurable with respect to $\mathcal{B}$. However, it is not a priori clear how this orthogonality can be used to deal with averages of the form $\Ex[f_1(L_1(\X)) \ldots f_m(L_m(\X))]$, for linear forms $L_1,\ldots,L_m$. The difficulty arises when one has to understand exponential averages such as $\Ex[\exp(P(X+Y)-P(X)-P(Y))]$. This average is $1$ when $P$ is a homogeneous polynomial of degree one, but it does not immediately follow from what we have said so far that it is small when $P$ is of higher degree. In this section we develop the results needed to deal with such exponential averages.

Consider a set of homogeneous polynomials $\{P_1,\ldots,P_k\}$, and a set of linear forms $\{L_1,\ldots,L_m\}$. We need to be able to analyze exponential averages of the form:
$$\Ex_{\X} \left[ \exp \left( \sum_{i=1}^k \sum_{j=1}^m  \lambda_{i,j} P_{i}(L_j(\X))\right)\right],$$
where $\lambda_{i,j} \in \F_p$. Lemma~\ref{lem:LinearFormsBias} below shows that if $\{P_1,\ldots,P_k\}$ are of sufficiently high rank, then it is either the case that $\sum_{i=1}^k \sum_{j=1}^m  \lambda_{i,j} P_{i}(L_j(\x)) \equiv 0$, which implies that the corresponding exponential average is exactly $1$, or otherwise the exponential average is very small. Note that this is  an ``approximate'' version of the case of characters. Namely if $\{\chi_{y_1},\ldots,\chi_{y_k}\}$ are characters of $\F_p^n$, then $\Ex\left[ \prod_{i=1}^k \prod_{j=1}^m \chi_{y_i}(L_j(\X)) \right]$ is either $1$ or $0$. In the case of polynomials of high rank, the ``zero'' case is approximated by a small number.

\begin{lemma}
\label{lem:LinearFormsBias}
Fix $p$ prime and $d < p$. Let $\{L_1,\ldots,L_m\}$ be a system of linear forms. Let $\P=\{P_{1},\ldots,P_k\}$ be a collection of homogeneous polynomials of degree at most $d$, such that $\rank(\P) > r_{p,d}(\eps)$. For every set of coefficients $\Lambda = \{\lambda_{i,j} \in \F_p: i \in [k], j\in [m] \}$, and
$$
P_\Lambda(\x) := \sum_{i=1}^k \sum_{j=1}^m  \lambda_{i,j} P_{i}(L_j(\x)),
$$
one of the following two cases holds:
$$\mbox{$P_\Lambda \equiv 0$ \qquad or \qquad $\bias(P_\Lambda) < \eps$}.$$
\end{lemma}

The proof of Lemma~\ref{lem:LinearFormsBias} is given in Section~\ref{sec:proofofLemmas}. Lemma~\ref{lem:LinearFormsBias} shows that in order to estimate $\bias(P_\Lambda)$ for polynomials of high rank $\P=\{P_{1},\ldots,P_k\}$, it suffices to determine  whether $P_\Lambda$ is identically $0$ or not. Our next observation, whose proof is given in Section~\ref{sec:proofofLemmas}, says that when the polynomials are homogeneous and linearly independent, then $P_\Lambda \equiv 0$ depends only on the set of the coefficients $\lambda_{i,j}$, the linear forms $L_j$, and the \emph{degrees} of the polynomials involved in $P_1,\ldots,P_k$, and not the particular choice of the polynomials.


\begin{lemma}
\label{lem:dependencyShape}
Let $\{L_1,\ldots,L_m\}$ be a system of linear forms over $\F_p^n$,  $\lambda_{i,j} \in \F_p$ for $i \in [k], j\in [m]$, and $d_1,\ldots,d_k \in [d]$. Then one of the following two cases holds:

\begin{enumerate}
\item[(i)] For every collection of linearly independent homogeneous polynomials $P_{1},\ldots,P_k$ of degrees  $d_1,\ldots,d_k$:
$$
\sum_{i=1}^k \sum_{j=1}^m  \lambda_{i,j} P_{i}(L_j(\x)) \equiv 0
$$

\item[(ii)] For every collection of linearly independent homogeneous polynomials $P_{1},\ldots,P_k$ of degrees $d_1,\ldots,d_k$:
$$
\sum_{i=1}^k \sum_{j=1}^m  \lambda_{i,j} P_{i}(L_j(\x)) \not\equiv 0
$$
\end{enumerate}
\end{lemma}

Sometimes one needs to deal with non-homogeneous polynomials. This for example is the case in our subsequent paper~\cite{TestingPaper} about correlation testing for affine invariant properties on $\F_p^n$. The following lemma shows that if the set of the linear forms is homogeneous, then the case of the non-homogeneous polynomials reduces to the homogeneous polynomials. Let $P:\F_p^n \rightarrow \F_p$ be a polynomial of degree at most $d$. For $1 \le l \le d$, let $P^{(l)}$ denote the homogeneous polynomial that is obtained from $P$ by removing all the monomials whose degrees are \emph{not} equal to $l$.

\begin{lemma}
\label{lem:nonhomogen}
Let $\mathcal{L}=\{L_1,\ldots,L_m\}$ be a \emph{homogeneous} system of linear forms. Furthermore, let $P_1,\ldots,P_k$ be polynomials of degrees  $d_1,\ldots,d_k$ such that $\{P_i^{(d_i)}: i \in [k]\}$ are linearly independent over $\F_p$. Then for  $\lambda_{i,j} \in \F_p$,
$$\mbox{$\sum_{i=1}^k \sum_{j=1}^m \lambda_{i,j} P_i(L_j(\x)) \equiv 0 \Longleftrightarrow \sum_{i=1}^k \sum_{j=1}^m  \lambda_{i,j} P_i^{(d_i)}(L_j(\x)) \equiv 0$}  $$
\end{lemma}

Using Lemma~\ref{lem:nonhomogen} one can derive the following analog of Lemma~\ref{lem:LinearFormsBias} for averages of nonhomogeneous polynomials when the system of linear forms is homogeneous.

\begin{lemma}
\label{lem:LinearFormsBias-nonhomogeneous}
Fix $p$ prime and $d < p$. Let $\{L_1,\ldots,L_m\}$ be a homogeneous system of linear forms. Let $\P=\{P_{1},\ldots,P_k\}$ be a collection of polynomials of degree at most $d$, such that $\rank(\P) > r_{p,d}(\eps)$. For every set of coefficients $\Lambda = \{\lambda_{i,j} \in \F_p: i \in [k], j\in [m] \}$, and
$$
P_\Lambda(\x) := \sum_{i=1}^k \sum_{j=1}^m  \lambda_{i,j} P_{i}(L_j(\x)),
$$
one of the following two cases holds:
$$\mbox{$P_\Lambda \equiv 0$ \qquad or \qquad $\bias(P_\Lambda) < \eps$}.$$
\end{lemma}

Lemmas~\ref{lem:LinearFormsBias},~\ref{lem:dependencyShape},~\ref{lem:nonhomogen}~and~\ref{lem:LinearFormsBias-nonhomogeneous} show that when studying the averages defined by systems of linear forms, under some homogeneity conditions (either for polynomials or for system of linear forms), high rank polynomials of the same degree sequence behave in a similar manner. The following proposition captures this.

\begin{proposition}[An invariance result]
\label{prop:invariance}
Let $\P=\{P_1,\ldots,P_k\},\Q=\{Q_1,\ldots,Q_k\}$ be two collections of polynomials over $\F_p^n$ of degree at most $d<p$ such that $\deg(P_i)=\deg(Q_i)$ for every $1\le i \le k$. Let $\mathcal{L}=\{L_1,\ldots,L_m\}$ be a system of linear forms, and $\Gamma:\F_p^{k} \to \mathbb{D}$ be an arbitrary function. Define $f,g:\F_p^n \to \mathbb{D}$ by
$$
f(x) = \Gamma(P_1(x),\ldots,P_k(x))
$$
and
$$
g(x) = \Gamma(Q_1(x),\ldots,Q_k(x)).
$$
Let $r_{p,d}:(0,1] \rightarrow \mathbb{N}$ be as given in Lemmas~\ref{lem:LinearFormsBias} and Lemma~\ref{lem:LinearFormsBias-nonhomogeneous}.
Then for every $\eps>0$ if $\rank(\P),\rank(\Q) > r_{p,d}(\eps)$,  we have
$$
\left| t_\mathcal{L}(f) - t_\mathcal{L}(g) \right| \le 2 \eps p^{mk},
$$
provided that at least one of the following two conditions hold:
\begin{enumerate}
\item[(i)] The polynomials $P_1,\ldots,P_k$ and $Q_1,\ldots,Q_k$ are homogeneous.
\item[(ii)] The system of linear forms $\mathcal{L}$ is homogeneous.
\end{enumerate}
\end{proposition}
\begin{proof}
The Fourier expansion of  $\Gamma$ shows
$$
\Gamma(z(1),\ldots,z(k)) = \sum_{\gamma \in \F_p^k}\widehat{\Gamma}(\gamma) \exp\left(\sum_{i=1}^k \gamma(i) \cdot z(i) \right).
$$
We thus have
$$
\prod_{i=1}^m f(L_i(\x)) =
\sum_{\gamma_1,\ldots,\gamma_m \in \F_p^k} \widehat{\Gamma}(\gamma_1) \cdot \ldots \cdot \widehat{\Gamma}(\gamma_m) \exp\left(\sum_{i \in [k], j \in [m]}\gamma_{j}(i) \cdot P_{i}(L_j(\x))\right),
$$
and
$$
\prod_{i=1}^m g(L_i(\x)) =
\sum_{\gamma_1,\ldots,\gamma_m \in \F_p^k} \widehat{\Gamma}(\gamma_1) \cdot \ldots \cdot \widehat{\Gamma}(\gamma_m) \exp\left(\sum_{i \in [k], j \in [m]} \gamma_{j}(i) \cdot Q_{i}(L_j(\x))\right).
$$

By Lemma~\ref{lem:LinearFormsBias} (under condition $(i)$) or Lemma~\ref{lem:LinearFormsBias-nonhomogeneous} (under condition $(ii)$) we know that for every $\gamma_1,\ldots,\gamma_m \in \F_p^k$ each one of
the polynomials $\sum_{i \in [k], j \in [m]}\gamma_{j}(i) \cdot P_{i}(L_j(\x))$ and $\sum_{i \in [k], j \in [m]} \gamma_{j}(i) \cdot Q_{i}(L_j(\x))$  is either  zero, or has bias at most $\eps$. But Lemmas~\ref{lem:dependencyShape} and~\ref{lem:nonhomogen} show that
under each one of the Conditions (i) or (ii), since $\deg(P_i)=\deg(Q_i)$, we have that
$$\sum_{i \in [k], j \in [m]}\gamma_{j}(i) \cdot P_{i}(L_j(\x)) \equiv 0 \Longleftrightarrow \sum_{i \in [k], j \in [m]} \gamma_j(i) \cdot Q_{i}(L_j(\x)) \equiv 0.$$
Hence we have that
$$\left|\Ex_{\X} \left[\exp\left(\sum_{i \in [k], j \in [m]}\gamma_{j}(i) \cdot P_{i}(L_j(\X))\right)\right]- \Ex_{\X} \left[\exp\left(\sum_{i \in [k], j \in [m]} \gamma_j(i) \cdot Q_{i}(L_j(\X))\right)\right]\right| \le 2\eps,$$
and
$$
\left| \Ex_{\X}\left[\prod_{i=1}^m f(L_i(\X))\right] - \Ex_{\X}\left[\prod_{i=1}^m g(L_i(\X))\right] \right| \le 2 \eps \cdot \sum_{\gamma_1,\ldots,\gamma_m \in \F_p^k} |\widehat{\Gamma}(\gamma_1)| \ldots |\widehat{\Gamma}(\gamma_m)|.
$$
Since $\|\widehat{\Gamma}\|_\infty \le 1$, we conclude
$$
\left| \Ex_{\X}\left[\prod_{i=1}^m f(L_i(\X))\right] - \Ex_{\X}\left[\prod_{i=1}^m g(L_i(\X))\right] \right| \le 2 \eps p^{mk}.
$$
\end{proof}

\subsection{Proofs of Lemmas~\ref{lem:LinearFormsBias},~\ref{lem:dependencyShape},~\ref{lem:nonhomogen}~and~\ref{lem:LinearFormsBias-nonhomogeneous}}\label{sec:proofofLemmas}}
We need to introduce several new notations in this section. We shall try to provide various examples to illustrate these concepts.

Let $P(x)$ be a homogeneous polynomial of degree $d<p$. Let $B(x_1,\ldots,x_d)$ be the symmetric multi-linear form associated with $P$; that is $P(x)=B(x,\ldots,x)$ and $B(x_1,\ldots,x_d) \equiv B(x_{\sigma_1},\ldots,x_{\sigma_d})$, for every permutation $\sigma$ of $[d]$.

\begin{example}
Consider a prime $p \ge 5$ and the polynomial $P(x):=6x(1)x(2)^2 \in \mathbb{F}_p[x(1),\ldots,x(n)]$.  Then $$B(x_1,x_2,x_3):=x_1(1)x_2(2)^2+x_1(2)x_2(1)^2 +x_1(1)x_3(2)^2+x_1(2)x_3(1)^2+x_2(1)x_3(2)^2+x_2(2)x_3(1)^2$$ is the symmetric  multi-linear form associated with $P$.
\end{example}

If $L(\x)=\sum_{i=1}^k c_i x_i$ is a linear form in $k$ variables, then we have
$$
P(L(\x)) = \sum_{i_1,\ldots,i_d \in [k]} c_{i_1} \ldots c_{i_d} B(x_{i_1},\ldots,x_{i_d}).
$$
For $\uu=(u_1,\ldots,u_d) \in [k]^d$, denote $\x_{\uu} = (x_{u_1},\ldots,x_{u_d})$. Let $U^d \subseteq [k]^d$ be defined as $U^d = \{(u_1,\ldots,u_d) \in [k]^d: u_1 \le u_2 \le \ldots \le u_d\}$. For $\uu \in U^d$, denote by $\ell_d(\uu)$ the number of distinct permutations\footnote{If the multiplicities of the elements of a multi-set are $i_1,\ldots,i_{\ell}$, then the number of distinct permutations of those elements is $\frac{(i_1+\ldots+i_{\ell})!}{i_1!\ldots i_{\ell}!}$.} of $(u_1,\ldots,u_d)$, and let $c_d(\uu,L):=c_{u_1}\ldots c_{u_d}$. Since $B$ is symmetric,  we have
\begin{equation}
\label{eq:PolOfLinear}
P(L(\x)) = \sum_{\uu \in U^d}  \ell_d(\uu) c_d(\uu,L) B(\x_{\uu}).
\end{equation}
Note that $\ell_d(\uu)$ depends only on $\uu$ and $c_d(\uu,L)$ depends only on the linear form $L$ and $\uu \in [k]^d$.

\begin{example}
\label{ex:notation}
Suppose that $p \ge 5$ is a prime and $P(x)$ is a homogenous polynomial of degree $3$. Let $B(x_1,x_2,x_3)$ be the symmetric  multi-linear form associated with $P$. Let $L(\x)=x_1+j x_2$ be a linear form in two variables where $j \in \mathbb{F}_p$ is a constant. In this case (\ref{eq:PolOfLinear}) becomes
\begin{eqnarray*}
P(x_1+j x_2) &=& B(x_1+jx_2,x_1+jx_2,x_1+jx_2) \\ &=& B(x_1,x_1,x_1) + j B(x_1,x_1,x_2)+ \ldots + j^3 B(x_2,x_2,x_2)\\
&=& B(x_1,x_1,x_1) + 3j B(x_1,x_1,x_2) + 3j^2 B(x_1,x_2,x_2)  +j^3 B(x_2,x_2,x_2).
\end{eqnarray*}
Note that for example the coefficient of $B(x_1,x_2,x_2)$ is $3j^2$ which is consistent with (\ref{eq:PolOfLinear}) as $\ell_3(1,2,2)=\frac{3!}{1! 2!}=3$ and $c_3((1,2,2),L)=1 \times j \times j=j^2$.
\end{example}

We need the following claim.

\begin{claim}
\label{claim:coeffs}
Let $B$ be a homogeneous multi-linear form over $\F_p$ of degree $d<p$. Consider a linear combination
$$
Q(\x) = \sum_{\uu \in U^d} c_{\uu} B(\x_{\uu})
$$
where not all the coefficients $c_{\uu}$ are zero. Then there exist $a_1,\ldots,a_k \in \F_p$ and $\alpha \in \mathbb{F}_p \setminus \{0\}$ such that for every $w \in \F_p^n$
$$
Q(a_1 w,\ldots, a_k w) = \alpha B(w,\ldots,w).
$$
\end{claim}
\begin{proof}
Consider $\x = (a_1 w,\ldots,a_k w)$. As $B$ is multi-linear,  we have $$B(\x_{\uu}) = \left( \prod_{i=1}^d a_{u_i} \right) \cdot B(w,\ldots,w),$$ for every $\uu \in U^d$. Let $a=(a_1,\ldots,a_k)$ and let $a^{\uu}$ denote the monomial $a^{\uu} = \prod_{i=1}^d a_{u_i}$. We thus have
$$
Q(a_1 w,\ldots,a_k w) = \left( \sum_{\uu \in U^d} c_{\uu} a^{\uu} \right) \cdot B(w,\ldots,w).
$$
Consider $g(a_1,\ldots,a_k)=\sum_{\uu \in U^d} c_{\uu} a^{\uu}$. This is a polynomial in $a_1,\ldots,a_k$ which is not identically
zero, as distinct $\uu \in U^d$ correspond to distinct monomials $a^{\uu}$. Hence there exists some assignment for $a$ for which $\alpha:=g(a) \ne 0$.
\end{proof}

Consider a set of linearly independent homogeneous polynomials $\{P_1,\ldots,P_k\}$, a system of linear forms $\{L_1,\ldots,L_m\}$ and some coefficients $\lambda_{i,j} \in \F_p$ where $i \in [m], j \in  [k]$. Let $B_{i}$ be the symmetric multi-linear form associated with $P_{i}$. Denoting $d_i:=\deg(P_i)$ and using the notation of (\ref{eq:PolOfLinear}), we thus have
\begin{eqnarray}
\label{eq:PlambdaExp}
\nonumber
P_{\Lambda}(\x) &=& \sum_{i=1}^k \sum_{j=1}^m \lambda_{i,j} P_i(L_j(\x))= \sum_{i=1}^k \sum_{j=1}^m \lambda_{i,j} \sum_{\uu \in U^{d_i}} \ell_{d_i}(\uu) c_{d_i}(\uu,L_j) B_i(\x_{\uu}) \\
&=& \sum_{i=1}^k  \sum_{\uu \in U^{d_i}} b^{d_i}_i(\uu) B_{i}(\x_{\uu}),
\end{eqnarray}
where
$$b_i^{d_i}(\uu):= \ell_{d_i}(\uu) \sum_{j=1}^m  \lambda_{i,j}  c_{d_i}(\uu,L_j).$$
Note that the coefficients $b_i^{d_i}(\uu)$ do not depend on the specific set of polynomials $P_1,\ldots,P_k$.

\begin{example}
Consider homogenous polynomials $P_1(x)$ of degree $1$ and $P_2(x)$ of degree $3$, and linear forms $L_j = x_1 + j x_2$ for $j=1,2,3,4$. Let $B_1(x_1)$ and $B_2(x_1,x_2,x_3)$ be the symmetric  multi-linear form associated with $P_1$ and $P_2$ respectively.
We have
$$P_1(x_1+j x_2) = B_1(x_1) + j B_1(x_2),$$
and as we saw in Example~\ref{ex:notation},
$$P_2(x_1+j x_2) = B_2(x_1,x_1,x_1) + 3j B_2(x_1,x_1,x_2)+3j^2 B_2(x_1,x_2,x_2) + j^3 B_2(x_2,x_2,x_2).$$
Now consider some coefficients $\lambda_{i,j} \in \mathbb{F}_p$ for $i \in [2]$ and $j \in [4]$. Then
\begin{eqnarray*}
P_{\Lambda}(\x) &=& \sum_{i=1}^2 \sum_{j=1}^4 \lambda_{i,j} P_i(L_j(\x)) = \left( \sum_{j=1}^4 \lambda_{1,j} \left(B_1(x_1) + j B_1(x_2) \right) \right)+
\\ && \left(\sum_{j=1}^4 \lambda_{2,j} \left( B_2(x_1,x_1,x_1) + 3j B_2(x_1,x_1,x_2)+3j^2 B_2(x_1,x_2,x_2) + j^3 B_2(x_2,x_2,x_2) \right)\right).
\end{eqnarray*}
Let us investigate the coefficient of the particular term $B_2(x_1,x_2,x_2)$ in this expression.
This coefficient is equal to $$\sum_{j=1}^4 \lambda_{2,j}  3 j^2 = 3 \sum_{j=1}^4 \lambda_{2,j} j^2 = \ell_{3}(1,2,2) \sum_{j=1}^4 \lambda_{2,j} c_{3}((1,2,2),L_j) = b_2^{3}(1,2,2),$$
as in (\ref{eq:PlambdaExp}).
\end{example}

Among Lemmas~\ref{lem:LinearFormsBias},~\ref{lem:dependencyShape},~\ref{lem:nonhomogen}~and~\ref{lem:LinearFormsBias-nonhomogeneous} first we prove Lemma~\ref{lem:dependencyShape}, which has the simplest proof.

\begin{proof}[Proof of Lemma~\ref{lem:dependencyShape}]
We will show that $P_{\Lambda}(\x) \equiv 0$ if and only if $b_i^{d_i}(\uu) = 0$, for all $i \in [k]$ and $\uu \in U^{d_i}$.
Let $B_{i_0}$ be a polynomial appearing in $P_{\Lambda}$ with a nonzero coefficient, that is  $b_{i_0}^{d_{i_0}}(\uu) \neq 0$ for some $\uu \in U^{d_{i_0}}$ in (\ref{eq:PlambdaExp}). Consider any assignment of the form $\x = (a_1 w, \ldots, a_k w)$. We have
that
$$
B_{i}(\x_{\uu}) = \a^{\uu} B_{i}(w,\ldots,w) = \a^{\uu} P_{i}(w).
$$
Hence we get that
\begin{equation}
\label{eq:eqqq}
P_{\Lambda}(a_1 w, \ldots, a_k w) = \sum_{i=1}^{k}  \alpha_{i}  P_{i}(w),
\end{equation}
where $\alpha_i = \sum_{\uu \in U^{d_i}} \a^{\uu} b_{i}^{d_i}(\uu)$. Applying Claim~\ref{claim:coeffs}, there exists a choice of $a_1,\ldots,a_k$, such that $\alpha_{i_0} \neq 0$, and then the linear independence of $\{P_1,\ldots,P_k\}$  shows that $P_\Lambda \not \equiv 0$.
\end{proof}

The proof of Lemma~\ref{lem:dependencyShape} shows that $P_{\Lambda}(\x) \not\equiv 0$ if and only if $b_i^{d_i}(\uu) \neq 0$ for some $i$ and $\uu$.
Next we prove Lemma~\ref{lem:LinearFormsBias} where we show that in this case under the stronger condition of high rank $\bias(P_\Lambda)$ is small.

\begin{proof}[Proof of Lemma~\ref{lem:LinearFormsBias}]
Suppose that  $P_{\Lambda}(\x) \not\equiv 0$ so that $b_i^{d_i}(\uu) \neq 0$, for some $i \in [k]$ and $\uu \in U^{d_i}$.  Let $d_{i_0} \le d$ be the largest degree such that $b_{i_0}^{d_{i_0}}(\uu) \neq 0$, for some $\uu \in U^{d_{i_0}}$.

Assume for contradiction that $\bias(P_{\Lambda}) \ge \eps$. By the regularity theorem for polynomials (Theorem~\ref{thm:regularity}) we get that $P_{\Lambda}(\x)$ can be expressed as a function of $r \le r_{p,d_{i_0}}(\eps) \le r_{p,d}(\eps)$ polynomials of degree at most $d_{i_0}-1$. We will show that this implies  $\rank(\P) \le r$. We know that
$P_{\Lambda}(\x)$ can be expressed as a function of at most $r$ polynomials of degree at most $d_{i_0}-1$. This
continues to hold under any assignment $\x = (a_1 w, \ldots, a_k w)$. That is
$$P_{\Lambda}(a_1 w, \ldots, a_k w) = \sum_{i=1}^{k}  \alpha_{i}  P_{i}(w), $$
is a function of at most $r$ polynomials of degree at most $d_{i_0}-1$, where $\alpha_i = \sum_{\uu \in U^{d_i}} \a^{\uu} b_{i}^{d_i}(\uu)$ .
By Claim~\ref{claim:coeffs} there exists a choice of $a_1,\ldots,a_k$ such that $\alpha_{i_0} \neq 0$, and this shows that $\rank(\P) \le r$.
\end{proof}

Next we prove Lemma~\ref{lem:nonhomogen} where we deal with the case that the polynomials are not necessarily homogeneous, but instead the system of linear forms $\{L_1,\ldots,L_m\}$ is homogeneous.

\begin{proof}[Proof of Lemma~\ref{lem:nonhomogen}]
It follows from Observation~\ref{obs:canonical} that by a change of variables we can assume that $L_1,\ldots,L_m$ are linear forms over $\F_p^n$ in $s$ variables $x_1,\ldots,x_s$, and  $x_1$ appears with coefficient $1$ in all linear forms. For $1\le i \le k$, let $B_{i}^1,\ldots,B_i^{d_i}$  be the symmetric multilinear forms associated with $P_{i}^{(1)},\ldots,P_i^{(d_i)}$, respectively. Then (\ref{eq:PlambdaExp}) must be replaced by
\begin{eqnarray}
\nonumber
P_{\Lambda}(\x) &=& \sum_{i=1}^k \sum_{j=1}^m  \lambda_{i,j} P_i(L_j(\x))=  \sum_{i=1}^k \sum_{j=1}^m \sum_{l=1}^{d_i} \lambda_{i,j} \sum_{\uu \in U^l} \ell_l(\uu) c_l(\uu,L_j) B_i^l(\x_{\uu}) \\  &=& \sum_{i=1}^k  \sum_{l=1}^{d_i} \sum_{\uu \in U^l} b_i^l(\uu) B^l_i(\x_{\uu}),
\label{eq:PlambdaExp2}
\end{eqnarray}
where
$$b_i^l(\uu):= \ell_l(\uu) \sum_{j=1}^m  \lambda_{i,j}  c_l(\uu,L_j).$$
Suppose that $P_{\Lambda}(\x) \not \equiv 0$. We claim that in this case there exists an $i \in [k]$ and $\uu \in U^{d_i}$ such that $b_i^{d_i}(\uu) \neq 0$, and this establishes the lemma. In order to prove this claim it suffices to show that if $b_i^{d_i}(\uu)=0$, for every $\uu \in  U^{d_i}$, then $b_i^t(\uu)=0$ for every $0 \le t \le d_i$ and $\uu \in U^t$. Let otherwise $t$ be the largest integer such that $b_i^t(\uu)= \ell_t(\uu) \sum_{j=1}^m  \lambda_{i,j}  c_t(\uu,L_j) \neq 0$, for some $\uu = (u_1,\ldots,u_t) \in U^t$. Consider $\uu'=(1,u_1,\ldots,u_t) \in U^{t+1}$. Since $x_1$ appears with coefficient $1$ in every $L_j$, we have that $c_{t+1}(\uu',L_j) = c_t(\uu,L_j)$. Also note that since $d_i < p$, and $p$ is a prime, $\ell_l(\uu) \neq 0$ for every $1 \le l \le d_i$ and $\uu \in U^l$. Hence we conclude that
$$b_i^{t+1}(\uu')= \ell_{t+1}(\uu') \sum_{j=1}^m  \lambda_{i,j} c_t(\uu,L_j) = \frac{\ell_{t+1}(\uu')}{\ell_{t}(\uu)} b_i^t(\uu) \neq 0,$$
which contradicts the maximality of $t$.
\end{proof}

We conclude with the proof of Lemma~\ref{lem:LinearFormsBias-nonhomogeneous}.

\begin{proof}[Proof of Lemma~\ref{lem:LinearFormsBias-nonhomogeneous}]
The proof is nearly identical to the proof of Lemma~\ref{lem:LinearFormsBias}. Let $P_1,\ldots,P_k$ be polynomials of degrees
$d_1,\ldots,d_k$. Decompose each polynomial to homogeneous parts $P_i^{(\ell)}$ for $1 \le \ell \le d_i$, and let $B_i^{\ell}$ denote the corresponding multi-linear polynomial. As in the proof of Lemma~\ref{lem:LinearFormsBias}, let $\ell$ be maximal such that $b_i^{\ell} \ne 0$ for some $i \in [k]$. As the system of linear forms is homogeneous we have by Lemma~\ref{lem:nonhomogen}
that also $b_i^{d_i} \ne 0$, hence $\ell=d_i$. The proof now continues exactly as in Lemma~\ref{lem:LinearFormsBias}.
\end{proof}

\section{Proof of the main theorem}
\label{sec:proof_main}

In this section we prove  Theorem~\ref{thm:strong_independence}. For the convenience of the reader, we restate the theorem.

\restate{Theorem~\ref{thm:strong_independence}}{
Let $\mathcal{L}=\{L_1,\ldots,L_m\}$ be a system of linear forms  of Cauchy-Schwarz complexity at most $p$. Let $d \ge 0$, and assume that $L_1^{d+1}$ is not in the linear span of $L_2^{d+1},\ldots,L_m^{d+1}$. Then for every $\eps>0$, there exists $\delta>0$ such that for any functions $f_1,\ldots,f_m:\F_p^n \to \D$ with $\|f_1\|_{U^{d+1}} \le \delta$, we have
$$
\left| \Ex_{\X \in (\F_p^n)^k} \left[\prod_{i=1}^m f_i(L_i(\X)) \right] \right| \le \eps,
$$
where $\X \in (\F_p^n)^k$ is uniform.
}

If $d \ge s$, then Theorem~\ref{thm:strong_independence} follows from Lemma~\ref{lemma:cauchy-schwarz-lemma}. Thus we assume $d<s$, hence also $d<p$. We will assume throughout the proof that $p,m,s,d,\eps$ are constants, and we will not explicitly state dependencies on them.

Let $r:\N \to \N$ be a growth function to be specified later. Let $\eta>0$ be a sufficiently small constant which will be specified later. By Lemma~\ref{lemma:decompose_multiple_func} there exists a polynomial factor $\mathcal{B}$ of degree $s$, complexity $C \le C_{\max}(\eta,r(\cdot))$ and rank at least $r(C)$ such that we can decompose each function $f_i$ as
$$
f_i = h_i + h'_i
$$
where $h_{i}=\E(f_i|\mathcal{B})$ and $\|h'_i\|_{U^{s+1}} \le \eta$. Trivially $\|h_{i}\|_{\infty} \le 1$ and $\|h'_{i}\|_{\infty} \le 2$. We first show that in order to bound $\Ex \left[\prod_{i=1}^m f_i(L_i(\X))\right]$ it suffices to bound $\Ex\left[ \prod_{i=1}^m h_i(L_i(\X))\right]$ if $\eta$ is chosen to be small enough.

\begin{claim}
If $\eta \le
\frac{\eps}{2m}$ then
$$
\left| \Ex\left[ \prod_{i=1}^m f_i(L_i(\X))\right] - \Ex \left[\prod_{i=1}^m h_{i}(L_i(\X))\right] \right| \le \eps/2,
$$
where the averages are over uniform $\X \in (\F_p^n)^k$.
\end{claim}

\begin{proof}
We have
$$
\prod_{i=1}^m f_i(L_i(\X)) - \prod_{i=1}^m h_{i}(L_i(\X)) =
\sum_{i=1}^m \left(\prod_{j=1}^{i-1} h_j (L_j(\X)) \cdot (f_i-h_i)(L_i(\X)) \cdot \prod_{j=i+1}^m f_j (L_j(\X))\right) .
$$
Fix $i \in [m]$. Since the Cauchy-Schwarz complexity of $\{L_1,\ldots,L_m\}$ is $s$, we have by Lemma~\ref{lemma:cauchy-schwarz-lemma} that
$$
\left|\Ex\left[\prod_{j=1}^{i-1} h_j (L_j(\X)) \cdot (f_i-h_i)(L_i(\X)) \cdot \prod_{j=i+1}^m f_j (L_j(\X))\right]\right| \le
\|f_i - h_i\|_{U^{s+1}} \le \eta.
$$
\end{proof}

%

%
We thus set $\eta=\frac{\eps}{2m}$, and regard $\eta$ from now on as a constant, and we do not specify explicitly dependencies on $\eta$ as well.

Let $\{P_i\}_{1 \le i \le C}$ be the polynomials which define the polynomial factor $\mathcal{B}$, where we assume each $P_i$ is homogeneous of degree $\deg(P_i) \le s$. Since each $h_{i}$ is measurable with regards to $\mathcal{B}$, we have $h_{i}(x)=\Gamma_i(P_1(x),\ldots,P_C(x))$ where $\Gamma_i:\F_p^C \to \D$ is some function. Decompose $\Gamma_i$ to its Fourier decomposition as
$$
\Gamma_i(z(1),\ldots,z(C)) = \sum_{\gamma \in \F_p^C} c_{i,\gamma} \cdot \exp\left(\sum_{j=1}^C \gamma(j) z(j)\right),
$$
where $|c_{i,\gamma}| \le 1$. Define for $\gamma \in \F_p^C$, the linear combination
$P_{\gamma}(x) = \sum_{j=1}^C \gamma(j) P_j(x)$. We can express each $h_{i}$ as
$$
h_{i}(x) = \sum_{\gamma \in \F_p^C} c_{i,\gamma} \cdot \exp\left(P_{\gamma}(x)\right),
$$
and we can express
\begin{equation}\label{eq:strong_independence:sum}
\Ex \left[ \prod_{i=1}^m h_{i}(L_i(\X))\right] = \sum_{\gamma_1,\ldots,\gamma_m  \in \F_p^C} \Delta(\gamma_1,\ldots,\gamma_m),
\end{equation}
where
\begin{equation}\label{eq:strong_independence:expression}
\Delta(\gamma_1,\ldots,\gamma_m) = \left(\prod_{i=1}^m c_{i,\gamma_i} \right)
\Ex \left[\exp\left(P_{\gamma_1}(L_1(\X))+\ldots+P_{\gamma_m}(L_m(\X))\right)\right].
\end{equation}
We will bound each term $\Delta(\gamma_1,\ldots,\gamma_m)$ by $\tau:=\tau(C)=p^{-mC} \eps/2$, which will establish the result. Let $S=\{\gamma \in \F_p^C: \deg(P_{\gamma}) \le d\}$. We first bound the terms $\Delta(\gamma_1,\ldots,\gamma_m)$ with $\gamma_1 \in S$.

\begin{claim}
\label{claim:bound_Delta_gamma1_in_S}
If the growth function $r(\cdot)$ is chosen large enough, and if $\delta>0$ is chosen small enough, then for all $\gamma_1 \in S$ we have
$$
|c_{1,\gamma_1}| \le \tau.
$$
Consequently, for all $\gamma_1 \in S$ and $\gamma_2,\ldots,\gamma_m \in \F_p^C$ we have
$$
|\Delta(\gamma_1,\ldots,\gamma_m)| \le \tau.
$$
\end{claim}

\begin{proof}
The bound on $|\Delta(\gamma_1,\ldots,\gamma_m)|$ follows trivially from the bound on $|c_{1,\gamma_1}|$, since
$|c_{2,\gamma_2}|,\ldots,|c_{m,\gamma_m}| \le 1$. To bound $|c_{1,\gamma_1}|$, note that
$$
c_{1,\gamma_1} = \Ex \left[ h_{1}(X) \exp(-P_{\gamma_1}(X))\right] - \sum_{\gamma' \in \F_p^C, \gamma' \ne \gamma_1} c_{1,\gamma'} \Ex \left[\exp(P_{\gamma'}(X) - P_{\gamma_1}(X))\right],
$$
where the averages are over uniform $X \in \F_p^n$. We first bound $\Ex \left[h_1(X) \exp(-P_{\gamma_1}(X))\right]$. Using the fact that $h_1=\E(f_1|\mathcal{B})$ and that the function $\exp(-P_{\gamma_1}(x))$ is $\mathcal{B}$-measurable, we have by Observation~\ref{obs:project_factor} that
$$
|\Ex \left[ h_{1}(X) \exp(-P_{\gamma_1}(X))\right]|=|\Ex \left[f_1(X) \exp(-P_{\gamma_1}(X))\right]| \le \|f_1\|_{U^{d+1}} \le \delta.
$$
Hence, by choosing $\delta<p^{-m  C_{\max}(r(\cdot))} \eps/4$ we guarantee that $|\Ex \left[ h_{1}(X) \exp(-P_{\gamma_1}(X))\right]|\le \delta<\tau/2$. Next for $\gamma' \neq \gamma_1$, we bound each term $\Ex\left[\exp(P_{\gamma'}(X) - P_{\gamma_1}(X))\right]$ by $\tau p^{-C}/2$. Assume that for some $\gamma' \ne \gamma_1$ we have $$
\left|\Ex \left[\exp(P_{\gamma'}(X) - P_{\gamma_1}(X))\right]\right| > \tau p^{-C}/2.
$$
Then by Theorem~\ref{thm:regularity} we have that
$$
\rank(P_{\gamma'} - P_{\gamma_1}) \le r_{p,s}(\tau p^{-C}/2)=r_1(C).
$$
Thus, as long as we choose $r(C) > r_1(C)$ for all $C \in \N$, we have that
$$
\sum_{\gamma' \ne \gamma_1} \left|c_{1,\gamma'} \Ex \left[\exp(P_{\gamma'}(X) - P_{\gamma_1}(X))\right]\right| \le \tau/2
$$
and we achieve the bound $|c_{1,\gamma_1}| \le \tau$.
\end{proof}

Consider now any $\gamma_1 \notin S$. We will show that if we choose $r(\cdot)$ large enough we can guarantee that
\begin{equation}
\label{eq:strong_independence:bias_linear_forms}
|\Ex \left[ \exp\left(P_{\gamma_1}(L_1(\X))+\ldots+P_{\gamma_m}(L_m(\X))\right)\right] | \le \tau,
\end{equation}
which will establish the result. Assume that this is not the case for some $\gamma_1 \notin S$ and $\gamma_2,\ldots,\gamma_m \in \F_p^C$. By Lemma~\ref{lem:LinearFormsBias} there
is a rank $r_{p,s}(\tau)=r_2(C)$ such that if we guarantee that $r(C)>r_2(C)$ for all $C \in \N$ and if~\eqref{eq:strong_independence:bias_linear_forms} does not hold, then we must have
\begin{equation}
\label{eq:strong_independence:equality_linear_forms}
P_{\gamma_1}(L_1(\x))+\ldots+P_{\gamma_m}(L_m(\x)) \equiv 0.
\end{equation}
Let $t=\deg(P_{\gamma_1})>d$. Let $P^{(t)}_{\gamma}$ be the degree $t$ homogeneous part of $P_{\gamma}$. Since the degrees of the polynomials are at most $p-1$, we must have that
\begin{equation}
\label{eq:strong_independence:equality_homogeneous_linear_forms}
P^{(t)}_{\gamma_1}(L_1(\x))+\ldots+P^{(t)}_{\gamma_m}(L_m(\x)) \equiv 0.
\end{equation}
The following claim concludes the proof. It shows that if~\eqref{eq:strong_independence:equality_homogeneous_linear_forms} holds, then $L_1^t$ is linearly dependent on $L_2^t,\ldots,L_m^t$. This immediately implies that also $L_1^{d+1}$ is linearly dependent on $L_2^{d+1},\ldots,L_m^{d+1}$ (since $t \ge d+1$) which contradicts our initial assumption.

\begin{claim}
Let $P_1,\ldots,P_m$ be homogeneous polynomials of degree $t<p$, where $P_1$ is not identically zero, such that
$$
P_1(L_1(\x))+\ldots+P_m(L_m(\x)) \equiv 0.
$$
Then $L_1^t$ is linearly dependent on $L_2^t,\ldots,L_m^t$.
\end{claim}

\begin{proof}
Let $M(x)=x_{i_1} \cdot \ldots \cdot x_{i_t}$ be a monomial appearing in $P_1$ with a nonzero coefficient $\alpha_1 \ne 0$. Let $\alpha_i$ be the coefficient of $M(x)$ in $P_i$ for $2 \le i \le m$. We have that
$$
\alpha_1 M(L_1(\x))+\ldots+\alpha_m M(L_m(\x)) \equiv 0.
$$
Let $\x=(x_1,\ldots,x_k)$ and $L_i(\x)=\lambda_{i,1}x_1 + \ldots + \lambda_{i,k} x_k$. We have
$$
M(L_i(\x)) = \prod_{j=1}^t (\lambda_{i,1}x_1(i_j) + \ldots + \lambda_{i,k} x_k(i_j)).
$$
Consider the assignment $x_i=(z(i),\ldots,z(i))$ where $z(1),\ldots,z(k) \in \F_p$ are new variables. We thus have the polynomial identity
$$
\sum_{i=1}^m \alpha_i (\lambda_{i,1}z(1) + \ldots + \lambda_{i,k} z(k))^t \equiv 0,
$$
which as $t<p$ is equivalent to
$$
\sum_{i=1}^m \alpha_i L_i^{t} \equiv 0.
$$
\end{proof}

\section{Summary and open problems}
\label{sec:summary}
We study the complexity of structures defined by linear forms.
Let $\mathcal{L}=\{L_1,\ldots,L_m\}$ be a system of linear forms of Cauchy-Schwarz complexity $s$. Our main technical contribution is that as long as $s \le p$ if $L_1^{d+1}$ is not linearly dependent on $L_2^{d+1},\ldots,L_m^{d+1}$,
then averages $\E[\prod_{i=1}^m f_i(L_i(\X))]$ are controlled by $\|f_1\|_{U^{d+1}}$.

When the first version of this article was submitted, Theorem~\ref{thm:inverse} in the case of $p \le d$ was still unknown. However as we mentioned in Section~\ref{sec:inverse}, very recently Tao and Ziegler~\cite{TaoZiegler2010} established this case. With the set of techniques of the present paper and~\cite{TaoZiegler2010}, it seems plausible that Theorem~\ref{thm:strong_independence} can be extended to Conjecture~\ref{conj:small_field} below which does not require any conditions on the Cauchy-Schwarz complexity of the system.  We leave this for future work.

\begin{conjecture}
\label{conj:small_field}
Fix a prime $p$ and $d \ge 1$. Let $\mathcal{L}=\{L_1,\ldots,L_m\}$ be a system of linear forms over $\F_p$ such that $L_1^{d+1}$ is not linearly dependent on $L_2^{d+1},\ldots,L_m^{d+1}$. Then for every $\eps>0$ there exists $\delta>0$ such that if $f_i:\F_p^n \to \D$ are functions such that $\|f_1\|_{U^{d+1}} \le \delta$ then
$$
\left| \E_{\X}\left[\prod_{i=1}^m f_i (L_i(\X))\right] \right| \le \eps.
$$
\end{conjecture}

\bibliographystyle{plain}
\bibliography{complexity}

\end{document}